\documentclass[reqno]{amsart}

\usepackage[english]{babel}

\usepackage{fullpage}
\usepackage{amsmath}
\usepackage{amsthm}
\usepackage{amssymb}
\usepackage{graphicx}
\usepackage{verbatim}
\usepackage[colorlinks=true, allcolors=blue]{hyperref}
\newtheorem{thm}{Theorem}[section]
\newtheorem{lem}[thm]{Lemma}
\newtheorem{prop}[thm]{Proposition}
\newtheorem{cor}[thm]{Corollary}
\newtheorem{defin}[thm]{Definition}
\theoremstyle{remark}
\newtheorem{rem}[thm]{Remark}
\numberwithin{equation}{section}

\renewcommand{\Re}{\operatorname{Re}}
\renewcommand{\Im}{\operatorname{Im}}

\title{Self-similar blowup for mass supercritical Schr\"odinger equations}

\author{Roland Donninger}
\email{roland.donninger@univie.ac.at}
\address{Universit\"at Wien, Fakult\"at f\"ur Mathematik,
  Oskar-Morgenstern-Platz 1, 1090 Vienna, Austria}

\author{Lorenz Lichtnecker}
\email{lorenz.lichtnecker@univie.ac.at}
\address{Universit\"at Wien, Fakult\"at f\"ur Mathematik,
  Oskar-Morgenstern-Platz 1, 1090 Vienna, Austria}
  
\thanks{This research was funded in whole or in part by the
Austrian Science Fund (FWF) 10.55776/P34560,
10.55776/PIN2161424, and 10.55776/PAT9429324. For open access purposes, the authors have
applied a CC BY public copyright license to any author-accepted manuscript version arising from this submission.}

\begin{document}

\maketitle

\begin{abstract}
We consider the focusing nonlinear Schr\"odinger equation in three spatial dimensions with powers close to three and prove the existence of a self-similar solution. This generalizes a previous result on the cubic case and shows that self-similar blowup is stable under perturbations of the equation.
\end{abstract}

\section{Introduction}
\noindent In this paper we consider the focusing, power-nonlinear Schrödinger equation (p-NLS)
\begin{equation}\label{p-NLS}
    i\partial_t\psi(t,x)+\Delta_x\psi(t,x)+\psi(t,x)|\psi(t,x)|^{p-1}=0    
\end{equation}
for an unknown $\psi:I\times\mathbb{R}^3\rightarrow\mathbb{C}$, where $I\subset\mathbb{R}$ is an interval and $p\in \left(\frac{7}{3},5\right)$ is mass supercritical and energy subcritical. Eq.~(\ref{p-NLS}) is the prototype for a nonlinear Schrödinger equation, describing the competing effects of dispersion and nonlinear focusing.

The Cauchy problem for Eq.~(\ref{p-NLS}) with sufficiently regular prescribed initial data $\psi(0,\cdot)=\psi_0:\mathbb{R}^3\rightarrow\mathbb{C}$ is well known to be locally well-posed, i.e., at least for short times $t$, there exists a unique solution that depends continuously on the initial data, see e.g.~\cite{Sul}. However, solutions may cease to exist after a finite time, see \cite{Gla}. There have been various attempts to quantify the behavior of solutions close to the time of breakdown and numerical simulations indicate the existence of a self-similar solution $\psi_{p,*}$ of the form
\begin{equation}\label{p-selfsimEq}
    \psi_{p,*}(t,x)=(2\alpha_p)^{-\tfrac{1}{p-1}}(1-t)^{-\tfrac{1}{p-1}-\tfrac{i}{2\alpha_p}}Q_p\left(\frac{x}{\sqrt{2\alpha_p(1-t)}}\right),
\end{equation}
defined on $(-\infty,1)\times\mathbb{R}^3$, where $Q_p:\mathbb{R}^3\rightarrow\mathbb{C}$ is a radial function in $L^{p+1}(\mathbb{R}^3)\cap\dot{H}^1(\mathbb{R}^3)\cap C^\infty(\mathbb{R}^3)\setminus \{0\}$, see \cite{Sul}. The true significance of this solution lies with its conjectured universality: Numerical simulations suggest that for sufficiently large but otherwise arbitrary initial data, the corresponding solution to Eq.~(\ref{p-NLS}) develops a singularity in finite time and that close to the time of breakdown, the shape of the solution approaches $\psi_{p,*}$ (modulo symmetries of the equation), independently of the concrete form of the initial data, see \cite{McLPapa, LandPapa, YangRoud}. In view of this universality, proving the existence of the self-similar solution $\psi_{p,*}$ is of great importance to understand the dynamics of Eq.~(\ref{p-NLS}). The first rigorous proof for existence of such a self-similar solution for the cubic case $p=3$ is due to the first author and Schörkhuber \cite{DS}, by implementing a Chebychev pseudo-spectral method combined with a contraction argument. Most importantly, the computer assistance used in \cite{DS} is completely elementary, solely based on integer arithmetic, and does not obscure the inherent argument of the proof. This provides sufficient flexibility to adapt the arguments for $p$ close to $3$ on a pure pen-and-paper basis, except for one estimate that requires computer assistance. A consequence of the results of the present paper is the observation that the existence of a self-similar solution to Eq.~(\ref{p-NLS}) is stable under perturbations of the exponent $p$. This shows that self-similar blowup is not restricted to the particular model with $p=3$ but generic among nonlinear Schr\"odinger equations in the mass supercritical regime. In particular, self-similar blowup is expected to be of fundamental relevance in the description of physical phenomena based on the nonlinear Schr\"odinger equation.

\begin{thm}[Main theorem, qualitative version]\label{qualThm}
There exists an open interval $\mathring{I}_*\subset\left(\tfrac{7}{3},5\right)$ such that $3\in \mathring{I}_*$ and for all $p\in I_*$, where $I_*$ denotes the closure of $\mathring{I}_*$, the following holds true. 
There exists a radial function $Q_p\in L^{p+1}(\mathbb{R}^3)\cap\dot{H}^1(\mathbb{R}^3)\cap C^\infty(\mathbb{R}^3)\setminus\{0\}$ and an $\alpha_p>0$ such that $\psi_{p,*}:(-\infty,1)\times\mathbb{R}^3\rightarrow\mathbb{C}$, defined by
$$\psi_{p,*}(t,x):=(2\alpha_p)^{-\tfrac{1}{p-1}}(1-t)^{-\tfrac{1}{p-1}-\tfrac{i}{2\alpha_p}}Q_p\left(\frac{x}{\sqrt{2\alpha_p(1-t)}}\right),$$
satisfies
$$i\partial_t\psi_{p,*}(t,x)+\Delta_x\psi_{p,*}(t,x)+\psi_{p,*}(t,x)|\psi_{p,*}(t,x)|^{p-1}=0$$
for all $(t,x)\in(-\infty,1)\times\mathbb{R}^3$. In particular, $Q_p$ satisfies the \emph{profile equation}
\begin{equation}\label{QEq}
\Delta Q_p(\xi)-Q_p(\xi)+i\alpha_p\left[\xi\cdot\nabla Q_p(\xi)+\tfrac{2}{p-1}Q_p(\xi)\right]+Q_p(\xi)|Q_p(\xi)|^{p-1}=0    
\end{equation}
for all $\xi\in\mathbb{R}^3$.
\end{thm}

In addition to the existence result in Theorem \ref{qualThm}, we obtain a precise approximation to $Q_p$ with rigorously proven error bounds.
\begin{defin}\label{PandaDef}
    For $n\in\mathbb{N}_0$ let $T_n:\mathbb{C}\rightarrow\mathbb{C}$ be the Chebychev polynomial of $n^\mathrm{th}$ degree, which is uniquely defined by $T_n(\cos(z))=\cos(nz)$ for all $z\in\mathbb{C}$. We define $P_*:(-1,1)\rightarrow\mathbb{C}$ by
    $$P_*(y):=\sum_{n=0}^{50}c_n(P_*)T_n(y)$$
    with $(c_n(P_*))_{n=0}^{50}\subset\mathbb{C}$ given in \cite{DS}, Appendix A.1,
    $$a_*:=\tfrac{772201763088846}{841768781900003},$$
    and $g_*:[0,\infty)\times\mathbb{R}\times\left(\tfrac{7}{3},5\right)\rightarrow\mathbb{C}$ given by
    $$g_*(r,a,p):=P_*\left(\frac{r-1}{r+1}\right)+\frac{2i(p-1)(a_*+a)}{(p-1)-2i(a_*+a)}P_*'(1)-P_*(1).$$
\end{defin}
\begin{thm}[Main theorem, quantitative version]\label{quantThm}
There exists an open interval $\mathring{I}_*\subset\left(\tfrac{7}{3},5\right)$ such that $3\in \mathring{I}_*$ and for all $p\in I_*$, where $I_*$ denotes the closure of $\mathring{I}_*$, the following holds true. 
There exists an $a_p\in[-10^{-10},10^{-10}]$ and a $g_p:[0,\infty)\rightarrow\mathbb{C}$ satisfying
$$2\sup_{r>0}r|g_p'(r)|+\sup_{r>0}|g_p(r)|\leq 1.2\cdot 10^{-6}$$
such that $Q_p:\mathbb{R}^3\rightarrow\mathbb{C}$, defined by
$$Q_p(x):=(1+|x|)^{-\tfrac{2}{p-1}-\tfrac{i}{a_*+a_p}}\left[g_*(|x|,a_p,p)+g_p(|x|)\right],$$
belongs to $L^{p+1}(\mathbb{R}^3)\cap\dot{H}^1(\mathbb{R}^3)\cap C^\infty(\mathbb{R}^3)\setminus\{0\}$ and 
$$\psi_{p,*}(t,x):=(2(a_*+a_p))^{-\tfrac{1}{p-1}}(1-t)^{-\tfrac{1}{p-1}-\tfrac{i}{2(a_*+a_p)}}Q_p\left(\frac{x}{\sqrt{2(a_*+a_p)(1-t)}}\right)$$
satisfies
$$i\partial_t\psi_{p,*}(t,x)+\Delta_x\psi_{p,*}(t,x)+\psi_{p,*}(t,x)|\psi_{p,*}(t,x)|^{p-1}=0$$
for all $(t,x)\in(-\infty,1)\times\mathbb{R}^3$.
\end{thm}

\subsection{Nature of the problem}
Concerning the properties of the blowup profile, it is worth noting that the flow of Eq.~(\ref{p-NLS}) preserves the $L^2$-norm and the energy
$$E[\psi(t,\cdot)]:=\frac{1}{2}\|\nabla\psi(t,\cdot)\|_{L^2(\mathbb{R}^d)}^2-\frac{1}{p+1}\|\psi(t,\cdot)\|_{L^{p+1}(\mathbb{R}^d)}^{p+1}=E[\psi(0,\cdot)],$$ 
where the former is invariant under the natural space-time scaling for $p=1+\tfrac{4}{d}$ (mass critical) and the latter for $p=1+\tfrac{4}{d-2}$ (energy critical). The cubic case $p=3$ in $3$ dimensions is mass supercritical and energy subcritical and so is the case $p$ close to $3$ studied in this paper. As a consequence, the self-similar profile cannot be in $L^2$ and must necessarily have vanishing energy, see e.g.~\cite{Sul}. Moreover, the asymptotic behavior of radial solutions to the profile equation Eq.~\eqref{QEq} is known \cite{Mesurier, Wang} and given by
$c_1Q_{p,1}+c_2Q_{p,2}$, with $c_1,c_2\in\mathbb{C}$ and
$$Q_{p,1}(\xi)\sim|\xi|^{-\tfrac{2}{p-1}-\tfrac{i}{\alpha_p}},\hspace{1cm}Q_{p,2}(\xi)\sim e^{-i\alpha_p\tfrac{\xi^2}{2}}|\xi|^{\tfrac{2}{p-1}-3+\tfrac{i}{\alpha_p}}$$
as $|\xi|\rightarrow\infty$.
In particular, a solution to Eq.~\eqref{QEq} has finite energy if and only if $c_2=0$, which is conjectured to hold only for specific values of $\alpha_p$. In other words, proving the existence of self-similar blowup amounts to solving a kind of nonlinear eigenvalue problem.

\subsection{Main idea of the proof}
In a naive attempt one could try to prove the result by an implicit function theorem argument, perturbing off the cubic case. 
This fails because of the phase invariance of Eq.~(\ref{QEq}), leading to a linearized operator with a nontrivial kernel. The key idea is to use the ``false inverse'' from \cite{DS} instead. Hence, we will employ the same construction for the linearized operator as in \cite{DS}, with slight adaptations depending on $p$. Arguing by continuity, we can reuse all of the bounds from \cite{DS}, using only one additional result obtained with computer assistance (see Lemma \ref{approxBD}). In particular, this shows that the proof in \cite{DS} is very robust and has clear conceptual advantages over increasingly popular ``brute force'' approaches where one makes standard numerical techniques rigorous by using off-the-shelf solvers based on interval arithmetic.

\subsection{Origin of the problem}
In order to put the problem into context, we briefly recall its origin in plasma physics.
In \cite{Zakh1, Zakh2}, Zakharov introduced a simplified dynamic description of an ionized plasma, based on averaging over the fast oscillation of the electric field. The cubic nonlinear Schrödinger equation, Eq.~(\ref{p-NLS}) with $p=3$, arises as the limit of the Zakharov system when the ion acoustic speed tends to infinity, which corresponds to the assumption that the plasma responds instantaneously to variations in the electric field. Finite-time blowup via radial self-similar solutions was proposed as a mechanism for energy dissipation and numerically confirmed by McLaughlin, Papanicolaou, Sulem, and Sulem \cite{McLPapa} for radial initial data. Moreover, these numerical simulations suggest the convergence to a universal blowup profile corresponding to a radial and (in absolute value) monotonically decreasing solution of Eq.~(\ref{QEq}) with $p=3$ and $\alpha\approx0,917$. The same universality was found even for nonradial data by Landman, Papanicolaou, Sulem, Sulem, and Wang \cite{LandPapa}. The properties of the profile strongly suggest that the self-similar solution constructed in \cite{DS} coincides with the one discovered in the numerical simulations.

\subsection{Related results}
The mathematical literature on the nonlinear Schr\"odinger equation is vast and we only mention a handful of results on blowup.

Self-similar blowup in the slightly mass supercritical case has been subject to various studies. For the cubic equation in dimension $0<d-2\ll 1$, Kopell and Landman \cite{KL11} as well as Rottsch\"afer and Kaper \cite{RK23} proved the existence of a unique $\alpha_3(d)$ and a corresponding radial finite energy self-similar profile. More generally, Bahri, Martel, and Rapha\"el \cite{BMR1} proved the existence of an $\alpha_p$ and a corresponding radial finite energy self-similar profile for $0<1+\tfrac{4}{d}-p\ll1$ in dimension $d\geq 1$, see also \cite{RK24, BCR3, B3, YangRoud}. In the parameter range $0<1+\tfrac{4}{d}-p\ll1$ and $1\leq d\leq5$, Merle, Rapha\"el, and Szeftel \cite{MRS17} constructed an open set of initial data in $H^1$ leading to finite-time self-similar blowup. The first rigorous proof of self-similar blowup far away from the critical case is the aforementioned \cite{DS}. Shortly after, another computer-assisted proof appeared \cite{DahFig24}. The approach in \cite{DahFig24} is very different from \cite{DS} in that it implements a shooting technique based on interval arithmetic.

A very different type of blowup is described by so-called ``standing ring solutions" of Eq.~(\ref{p-NLS}), where the solution blows up on a sphere. Such solutions have been constructed by Holmer, Roudenko \cite{HR9}, Rapha\"el \cite{R21}, and Rapha\"el, Szeftel \cite{RS22}. Fibich, Gabish, and Wang \cite{FGW5, FGW6} numerically observed yet another type of blowup for Eq.~(\ref{p-NLS}). These ``collapsing ring solutions" focus on a sphere that collapses towards the origin and have been proven to exist for $d\geq2$, $0<1+\tfrac{4}{d}-p\ll1$ by Merle, Raphael and Szeft\"el \cite{MRS18} and for $d=p=3$ by Holmer, Perelman, and Roudenko \cite{HPR8}.

We remark in passing that the mass critical case $p=1+\tfrac{4}{d}$ is much better understood. Singularity formation in finite time occurs there as well but the mechanism is fundamentally different and not self-similar, see e.g.~\cite{MRupper}.

\section{Setup of the problem}
\noindent For $p\in(\tfrac{7}{3},5)$ we rewrite Eq.~(\ref{p-NLS}) in similarity coordinates, i.e., we define a new unknown $u$ by
\begin{equation}
    \psi(t,x)=(2\alpha)^{-\tfrac{1}{p-1}}(1-t)^{-\tfrac{1}{p-1}}u\left(-\frac{1}{2\alpha}\log(1-t),\frac{x}{\sqrt{2\alpha(1-t)}}\right),    
\end{equation}
where $\alpha>0$, $t<1$ and $x\in\mathbb{R}^3$. Then $\psi$ satisfies Eq.~(\ref{p-NLS}) if and only if $u$ satisfies
\begin{equation}\label{p_NLS-simcoord}
    i\partial_\tau u(\tau,\xi)+\Delta_\xi u(\tau,\xi)+i\alpha\left[\xi\cdot\nabla_\xi u(\tau,\xi)+\frac{2}{p-1}u(\tau,\xi)\right]+u(\tau,\xi)|u(\tau,\xi)|^{p-1}=0
\end{equation}
for $\tau\in\mathbb{R}$ and $\xi\in\mathbb{R}^3$.
Inserting the ansatz $u(\tau,\xi)=e^{i\tau}Q(\xi)$ for a self-similar profile $Q:\mathbb{R}^3\rightarrow\mathbb{C}$ yields Eq.~(\ref{QEq}). By writing $Q(\xi)=q(|\xi|)$, Eq.~(\ref{p_NLS-simcoord}) transforms into
\begin{equation}\label{q-Eq}
    q''(r)+\left(\frac{2}{r}+i\alpha r\right)q'(r)+\left(\frac{2i\alpha}{p-1}-1\right)q(r)+q(r)|q(r)|^{p-1}=0
\end{equation}
for $r>0$.
\subsection{Compactification}
To overcome the difficulties of the unbounded domain and the oscillatory behavior near infinity of Eq.~(\ref{q-Eq}), we compactify the problem by setting
$$q(r)=(1+r)^{-\tfrac{2}{p-1}-\tfrac{i}{\alpha}}f\left(\frac{r-1}{r+1}\right).$$
Then $q$ satisfies Eq.~(\ref{q-Eq}) if and only if $\mathcal{R}(\alpha,p,f)=0$ on $(-1,1)$, where
$$\mathcal{R}(\alpha,p,f)(y):=f''(y)+p_0(y,\alpha,p)f'(y)+q_0(y,\alpha,p)f(y)+\frac{f(y)|f(y)|^{p-1}}{(1-y)^2},$$
and
\begin{align*}
    p_0(y,\alpha,p)&:=\frac{4i\alpha}{(1-y)^3}-\frac{2i\alpha}{(1-y)^2}-\frac{\frac{4}{p-1}+\frac{2i}{\alpha}}{1-y}+\frac{2}{1+y},\\
    q_0(y,\alpha,p)&:=\frac{\frac{4i\alpha}{p-1}-2}{(1-y)^3}+\frac{\frac{2}{p-1}\left(\frac{2}{p-1}-1\right)+\left(\frac{4}{p-1}-1\right)\frac{i}{\alpha}-\frac{1}{\alpha^2}}{(1-y)^2}-\frac{\frac{2}{p-1}+\frac{i}{\alpha}}{1-y}-\frac{\frac{2}{p-1}+\frac{i}{\alpha}}{1+y}.
\end{align*}
We fix the interval $$J_*:=[-10^{-10},10^{-10}]$$ for the parameter $a$ as in \cite{DS}, and let $\mathring{I}\subset(\tfrac{7}{3},5)$ be an open interval with $3\in \mathring{I}$ and denote by $I$ its closure.
For $a\in J_*$ and $p\in I$ we set
\begin{equation}\label{f*def}
f_*(y,a,p):=P_*(y)+\frac{2i(p-1)(a_*+a)}{(p-1)-2i(a_*+a)}P_*'(1)-P_*(1),    
\end{equation}
with $P_*:\mathbb{Q}\rightarrow\mathbb{Q}$ and $a_*\in\mathbb{Q}$ as in Definition \ref{PandaDef}. Observe that $f_*\in C^2([-1,1]\times J_*\times I)$ and that the $(a,p)$-dependent correction term accounts for the leading-order asymptotics. We claim that by choosing $I$ sufficiently small around $3$, for every $p\in I$ we can find a correction $(a,f)$ such that $\mathcal{R}(a_*+a,p,f_*(\cdot,a,p)+f)=0$. 
Linearizing in $f$ gives
$$\mathcal{R}(a_*+a,p,f_*(\cdot,a,p)+f)=\mathcal{R}(a_*+a,p,f_*(\cdot,a,p))+\mathcal{L}_{a,p}(f)+\mathcal{N}_{a,p}(f),$$
where $\mathcal{L}_{a,p}$ is $\mathbb{R}$-linear and defined by
\begin{align*}
 \mathcal{L}_{a,p}(f)(y):=&\ f''(y)+p_0(y,a_*+a,p)f'(y)+q_0(y,a_*+a,p)f(y)\\
 &+\frac{p+1}{2}\frac{|f_*(y,a,p)|^{p-1}}{(1-y)^2}f(y)+\frac{p-1}{2}\frac{|f_*(y,a,p)|^{p-3}f_*(y,a,p)^2}{(1-y)^2}\bar{f}(y),   
\end{align*}
and
\begin{align*}
    \mathcal{N}_{a,p}(f)(y):=&\ \frac{(f_*(y,a,p)+f(y))|f_*(y,a,p)+f(y)|^{p-1}}{(1-y)^2}\\&-\frac{p+1}{2}\frac{|f_*(y,a,p)|^{p-1}}{(1-y)^2}f(y)-\frac{p-1}{2}\frac{|f_*(y,a,p)|^{p-3}f_*(y,a,p)^2}{(1-y)^2}\bar{f}(y)
\end{align*}
contains the nonlinear terms in $f$. Consequently, the problem boils down to solving
\begin{equation}\label{probinLap}
\mathcal{L}_{a,p}(f)=-\mathcal{R}(a_*+a,p,f_*(\cdot,a,p))-\mathcal{N}_{a,p}(f).    
\end{equation}

\subsection{Strategy of the proof}
\begin{itemize}
    \item First, note that Eq.~(\ref{q-Eq}) is phase invariant: If $q$ is a solution then so is $e^{i\theta}q$ for any $\theta\in\mathbb{R}$. Consequently, the linearization has a nontrivial kernel and this is a severe difficulty.
As a first step, we will construct an approximate fundamental system for $\mathcal{L}_{a,p}$, which defines an approximate $\mathbb{R}$-linear operator $\tilde{\mathcal{L}}_{a,p}$ that is close to $\mathcal{L}_{a,p}$ in a suitable sense. By the variation of constants formula, we can then formally invert the operator $\tilde{\mathcal{L}}_{a,p}$, such that Eq.~(\ref{probinLap}) is formally equivalent to 
    \begin{equation}\label{ContrEq}
    f=\tilde{\mathcal{L}}_{a,p}^{-1}\left(\tilde{\mathcal{L}}_{a,p}(f)-\mathcal{L}_{a,p}(f)-\mathcal{N}_{a,p}(f)-\mathcal{R}(a_*+a,p,f_*(\cdot,a,p)\right)
\end{equation}
    However, while $\tilde{\mathcal{L}}_{a,p}$ may not have a nontrivial kernel, it will certainly have an eigenvalue very close to $0$. Thus, in the construction of the inverse we are forced to choose a fundamental system that involves a function that is bad at both endpoints. This means that 
    $\tilde{\mathcal{L}}_{a,p}(\tilde{\mathcal{L}}_{a,p}^{-1}(g))=g$ formally but $\tilde{\mathcal{L}}_{a,p}^{-1}$ does not map into the (yet to be defined) right space and is therefore not an inverse of $\tilde{\mathcal{L}}_{a,p}$ in the functional analytic sense.
    \item Next, we add a regularization term of the form $f_0\psi(a,p,g)$, which is essentially the projection of $\tilde{\mathcal{L}}_{a,p}^{-1}(g)$ onto the fundamental solution that is singular at $-1$ multiplied by a smooth cut-off function supported near $-1$. In this way, we obtain a new operator of the form
    $$\mathcal{J}(a,p,g)=\tilde{\mathcal{L}}_{a,p}^{-1}(g)+f_0\psi(a,p,g)$$
    which now maps into the right space but is no longer a formal inverse of $\tilde{\mathcal{L}}_{a,p}$, unless $\psi(a,p,g)=0$.
    \item By the contraction mapping principle we then solve
    $$f=\mathcal{J}\left(a,p,\tilde{\mathcal{L}}_{a,p}(f)-\mathcal{L}_{a,p}(f)-\mathcal{N}_{a,p}(f)-\mathcal{R}(a_*+a,p,f_*(\cdot,a,p)\right)$$
    and obtain a solution for all $a\in J_*$ and $p$ sufficiently close to $3$.
    \item Finally, by further shrinking $I$ if necessary, we prove the existence of a compact neighborhood $I_*$ of $3$ such that for all $p\in I_*$
    $$a\mapsto\psi\left(a,p,\tilde{\mathcal{L}}_{a,p}(f_{a,p})-\mathcal{L}_{a,p}(f_{a,p})-\mathcal{N}_{a,p}(f_{a,p})-\mathcal{R}(a_*+a,p,f_*(\cdot,a,p)\right)$$
    changes sign and hence, by continuity, vanishes at some $a$. In particular, $(a,f_{a,p})$ satisfy Eq.~(\ref{ContrEq}), or, equivalently,
    $$\mathcal{R}(a_*+a,p,f_*(\cdot,a,p)+f_{a,p})=0.$$
\end{itemize}
The proof extends the result of \cite{DS} which treats the case $p=3$. In \cite{DS}, explicit bounds for all of the terms appearing in the contraction argument are established using computer assistance.
The present paper requires only one additional computer assisted result (see Lemma \ref{approxBD}), since all the other bounds from \cite{DS} can be reused by continuity as long as $p$ is sufficiently close to $3$.

The main adaptations for the generalization in $p$ concern the definition of $f_*$, Eq.~(\ref{f*def}), and the approximate fundamental system for $\mathcal{L}_{a,p}$. However, as one might expect in view of Eq.~(\ref{QEq}), the oscillatory behavior of fundamental solutions is not affected by the parameter $p$.
\subsection{Function spaces}
To analyze the asymptotic behavior of solutions to the equation $\mathcal{L}_{a,p}(f)=0$, we consider the $\mathbb{C}$-linear linearization at $0$,
$$L_{\alpha,p}(f)(y):=f''(y)+p_0(y,\alpha,p)f'(y)+q_0(y,\alpha,p)f(y)=0,$$
which (for $\alpha=a_*+a$) to leading order has the same behavior near the endpoints as $\mathcal{L}_{a,p}(f)=0$.
The asymptotics at the irregularly singular endpoint $1$ can be determined by making the rough WKB ansatz $f(y)\sim e^{(1-y)^\sigma}$. Since $p_0(y,\alpha,p)\sim (1-y)^3$, we must have $\sigma=2$, which motivates the refined ansatz
$$f(y)\sim e^{a_2(1-y)^{-2}+a_1(1-y)^{-1}+a_0\log(1-y)}$$
and by comparing orders we determine the coefficients 
$(a_2,a_1,a_0)\in\lbrace(-2i\alpha,2i\alpha,3-\tfrac{4}{p-1}-\tfrac{2i}{\alpha}),(0,0,0)\rbrace$.
Consequently, near $1$, we expect a singular oscillatory solution $\widehat{f_1}(y,\alpha,p)\sim(1-y)^{3-\tfrac{4}{p-1}}e^{i\varphi(y,\alpha)}$ near the endpoint $1$, where
$$\varphi(y,\alpha)=-\frac{2\alpha}{(1-y)^2}+\frac{2\alpha}{1-y}-\frac{2}{\alpha}\log(1-y)$$
and a smooth solution $f_1(y,\alpha,p)\sim 1$.

Near -1, the situation is more classical because the singularity at $y=-1$ is regular and we find two solutions $f_{-1}(y,\alpha,p)\sim 1$ and $\widehat{f_{-1}}(y,\alpha,p)\sim (1+y)^{-1}$ by Frobenius analysis.
For the Wronskian we obtain the expression
\begin{align*}
    W(f_{-1}(\cdot,\alpha,p),f_1(\cdot,\alpha,p))(y)&=W(f_{-1}(\cdot,\alpha,p),f_1(\cdot,\alpha,p))(0)e^{-\int_0^yp_0(x,\alpha,p)dx}\\
    &=W(f_{-1}(\cdot,\alpha,p),f_1(\cdot,\alpha,p))(0)(1-y)^{-\tfrac{4}{p-1}}(1+y)^{-2}e^{i\varphi(y,\alpha)}=:W_*(y,\alpha,p).
\end{align*}
Assuming that $W(f_{-1}(\cdot,\alpha,p),f_1(\cdot,\alpha,p))(0)\neq 0$, the variation of constants formula yields the inverse operator
$$L^{-1}_{\alpha,p}(g)(y)=f_{-1}(y,\alpha,p)\int_y^1\frac{f_1(x,\alpha,p)}{W_*(x,\alpha,p)}g(x)dx+f_1(y,\alpha,p)\int_{-1}^y\frac{f_{-1}(x,\alpha,p)}{W_*(x,\alpha,p)}g(x)dx.$$
From
\begin{align*}
    |f_{-1}(y,\alpha,p)|&\lesssim 1,&&|f_{-1}'(y,\alpha,p)|\lesssim(1-y)^{-\tfrac{4}{p-1}},\\
    |f_{1}(y,\alpha,p)|&\lesssim (1+y)^{-1},&&\ \ |f_{1}'(y,\alpha,p)|\lesssim(1+y)^{-2},
\end{align*}
we infer the bounds
\begin{align*}
    |L^{-1}_{\alpha,p}(g)(y)|&\lesssim\int_y^1(1+x)(1-x)^{\tfrac{4}{p-1}}|g(x)|dx+(1+y)^{-1}\int_{-1}^y(1+x)^2(1-x)^{\tfrac{4}{p-1}}|g(x)|dx\\
    &\lesssim\sup_{y\in(-1,1)}(1+y)(1-y)^2|g(y)|\int_{-1}^1(1-x)^{\tfrac{4}{p-1}-2}dx\\
    &\lesssim\sup_{y\in(-1,1)}(1+y)(1-y)^2|g(y)|
\end{align*}
and
\begin{align*}
    (1-y^2)|L^{-1}_{\alpha,p}(g)'(y)|\lesssim&\ (1-y)^{1-\tfrac{4}{p-1}}\int_y^1(1+x)(1-x)^{\tfrac{4}{p-1}}|g(x)|dx\\&+(1+y)^{-1}\int_{-1}^y(1+x)^2(1-x)^{\tfrac{4}{p-1}}|g(x)|dx\\
    \lesssim&\sup_{y\in(-1,1)}(1+y)(1-y)^2|g(y)|\\&\times\left((1-y)^{1-\tfrac{4}{p-1}}\int_y^1(1-x)^{\tfrac{4}{p-1}-2}dx+\int_{-1}^y(1-x)^{\tfrac{4}{p-1}-2}dx\right)\\
    \lesssim&\sup_{y\in(-1,1)}(1+y)(1-y)^2|g(y)|
\end{align*}
for all $p\in\left(\tfrac{7}{3},5\right)$. Motivated by this, we define $X$ as the space of all $f\in C([-1,1])\cap C^1(-1,1)$ such that $y\mapsto(1-y^2)f'(y)$ extends to a continuous function on $[-1,1]$ and $Y$ as the space of all $f\in C(-1,1)\setminus\lbrace 0\rbrace$ such that $y\mapsto(1+y)(1-y)^2f(y)$ extends to a bounded continuous function on $[-1,1]\setminus\lbrace 0\rbrace$.
Defining the norms
\begin{align*}
   \|f\|_X&:=\sup_{y\in(-1,1)}(1-y^2)|f'(y)|+\|f\|_{L^\infty(-1,1)},\\\|f\|_Y&:=\sup_{y\in(-1,1)}(1+y)(1-y)^2|f(y)|,
\end{align*}
we obtain the Banach spaces $(X,\|.\|_X)$ and $(Y,\|.\|_Y)$ and the bound
$$\|L^{-1}_{\alpha,p}(g)\|_X\lesssim\|g\|_Y.$$
Note that in general $\mathcal{R}(\alpha,p,f)$ is cubically singular at $1$ and hence does not belong to $Y$. The $(a,p)$-dependent correction term in the definition of $f_*$ (\ref{f*def}) serves to cancel one order of this singularity.
\begin{lem}\label{Rcontlem}
We have $\mathcal{R}(a_*+a,p,f_*(\cdot,a,p))\in Y$ for all $(a,p)\in J_*\times I$. Furthermore, the map
$$(a,p)\mapsto\mathcal{R}(a_*+a,p,f_*(\cdot,a,p)):J_*\times I\rightarrow Y$$
is continuous.
\end{lem}
\begin{proof}
    Let $(a,p)\in J_*\times I$. Then\footnote{A prime denotes the derivative with respect to the first variable, i.e., $f_*'(y,a,p):=\partial_yf_*(y,a,p)$.}
    $$\lim_{y\rightarrow1}(1-y)^3\mathcal{R}(a_*+a,p,f_*(\cdot,a,p))(y)=4i(a_*+a)f_*'(1,a,p)+\left(\frac{4i(a_*+a)}{p-1}-2\right)f_*(1,a,p)=0$$
    and $f_*(\cdot,a,p)\in C^2([-1,1])$ show that the one-sided limits $\lim_{y\rightarrow\pm1}(1+y)(1-y)^{2}\mathcal{R}(a_*+a,p,f_*(\cdot,a,p))\in\mathbb{C}$ exist and depend continuously on $(a,p)$. In particular, the map
    $$(y,a,p)\mapsto(1+y)(1-y)^2\mathcal{R}(a_*+a,p,f_*(\cdot,a,p))(y)$$
    extends continuously to $[-1,1]\times J_*\times I$ and is thus uniformly continuous in $y$.
\end{proof}
\section{Approximation of the operator $\mathcal{L}_{a,p}$}
\noindent Since the operator $\mathcal{L}_{a,p}$ is not $\mathbb{C}$-linear, but only $\mathbb{R}$-linear, the real- and imaginary parts of the fundamental system need to be separated, leading to a system of $4$ fundamental solutions. We will use an approximate left fundamental system adapted to the boundary conditions at $-1$, which are independent of $a$ and $p$, and an approximate right fundamental system adapted to the $(a,p)$-dependent boundary conditions at $1$ to sufficient order. The approximate left and right fundamental system are then being connected at $0$ by using a suitable $\mathbb{R}^{4\times 4}$ matrix transformation.

For $y\in(-1,0]$ we set
\begin{align*}
    f_{L,1}(y)&:=\ 1+(1+y)P_{L,1}(y), &&f_{L,2}(y):=i+(1+y)P_{L,2}(y),\\
    f_{L,3}(y)&:=\frac{1}{1+y}\left[1+(1+y)^2P_{L,3}(y)\right],&&f_{L,4}(y):=\frac{i}{1+y}\left[1+(1+y)^2P_{L,4}(y)\right],
\end{align*}
where
$$P_{L,j}(y):=\sum_{n=0}^{16}c_n(P_{L,j})T_n(2y+1),\hspace{1cm}j\in\lbrace 1,2,3,4\rbrace,$$
with $(c_n(P_{L,j}))_{n=0}^{16}\subset\mathbb{C}$ given in \cite{DS}, Appendix A.2,
and
$$F_L:=\begin{pmatrix}\Re f_{L,1}&\Re f_{L,2}&\Re f_{L,3}&\Re f_{L,4}\\
\Im f_{L,1}&\Im f_{L,2}&\Im f_{L,3}&\Im f_{L,4}\\
\Re f'_{L,1}&\Re f'_{L,2}&\Re f'_{L,3}&\Re f'_{L,4}\\
\Im f'_{L,1}&\Im f'_{L,2}&\Im f'_{L,3}&\Im f'_{L,4}
\end{pmatrix}.$$
\\
For $(y,a,p)\in[0,1)\times J_*\times I$ we set
\begin{align*}
    f_{R,1}(y,a,p):=&\ 1+\left(\frac{a_*+a+i}{2(a_*+a)}+\frac{3-p}{2(p-1)}\right)(1-y)+(1-y)^2P_{R,1}(y),\\
    f_{R,2}(y,a,p):=&\ i+i\left(\frac{a_*+a+i}{2(a_*+a)}+\frac{3-p}{2(p-1)}\right)(1-y)+(1-y)^2P_{R,2}(y),\\
    f_{R,3}(y,a,p):=&\ e^{i\varphi(y,a_*+a)}(1-y)^{3-\tfrac{4}{p-1}}\left[1+\frac{2(a_*+a)-i}{2(a_*+a)}(1-y)+(1-y)^2P_{R,3}(y)\right]\\
    &+e^{-i\varphi(y,a_*+a)}(1-y)^{7-\tfrac{4}{p-1}}Q_{R,3}(y),\\
    f_{R,4}(y,a,p):=&\ ie^{i\varphi(y,a_*+a)}(1-y)^{3-\tfrac{4}{p-1}}\left[1+\frac{2(a_*+a)-i}{2(a_*+a)}(1-y)+(1-y)^2P_{R,4}(y)\right]\\
    &+ie^{-i\varphi(y,a_*+a)}(1-y)^{7-\tfrac{4}{p-1}}Q_{R,4}(y),
\end{align*}
where
\begin{align*}
    P_{R,j}(y)&:=\sum_{n=0}^{30}c_n(P_{R,j})T_n(2y-1),\hspace{1cm}j\in\lbrace 1,2,3,4\rbrace,\\
    Q_{R,j}(y)&:=\sum_{n=0}^{26}c_n(Q_{R,j})T_n(2y-1),\hspace{1cm}j\in\lbrace 3,4\rbrace,
\end{align*}
with $(c_n(P_{R,j}))_{n=0}^{30}\subset\mathbb{C}$ and $(c_n(Q_{R,j}))_{n=0}^{26}\subset\mathbb{C}$ given in \cite{DS}, Appendix A.3,
and
$$F_R(y,a,p):=\begin{pmatrix}\Re f_{R,1}(y,a,p)&\Re f_{R,2}(y,a,p)&\Re f_{R,3}(y,a,p)&\Re f_{R,4}(y,a,p)\\
\Im f_{R,1}(y,a,p)&\Im f_{R,2}(y,a,p)&\Im f_{R,3}(y,a,p)&\Im f_{R,4}(y,a,p)\\
\Re f'_{R,1}(y,a,p)&\Re f'_{R,2}(y,a,p)&\Re f'_{R,3}(y,a,p)&\Re f'_{R,4}(y,a,p)\\
\Im f'_{R,1}(y,a,p)&\Im f'_{R,2}(y,a,p)&\Im f'_{R,3}(y,a,p)&\Im f'_{R,4}(y,a,p)
\end{pmatrix}.$$\\
Finally, for $(y,a,p)\in(-1,1)\times J_*\times I$ we set
$$F(y,a,p):=1_{(-1,0]}(y)F_L(y)M(a,p)+1_{(0,1)}(y)F_R(y,a,p)$$
with
$$M(a,p):=F_L^{-1}(0)F_R(0,a,p),$$ and\footnote{For $A\in\mathbb{C}^{n\times n}$, we denote by $A_k$ the $k$-th column vector and by $A_k^j$ the element in the $j$-th row and $k$-th column.}
$$f_j(y,a,p):=F(y,a,p)_j^1+iF(y,a,p)_j^2.$$
Then we have
$$F(y,a,p)=\begin{pmatrix}\Re f_1(y,a,p)&\Re f_2(y,a,p)&\Re f_3(y,a,p)&\Re f_4(y,a,p)\\
\Im f_1(y,a,p)&\Im f_2(y,a,p)&\Im f_3(y,a,p)&\Im f_4(y,a,p)\\
\Re f'_1(y,a,p)&\Re f'_2(y,a,p)&\Re f'_3(y,a,p)&\Re f'_4(y,a,p)\\
\Im f'_1(y,a,p)&\Im f'_2(y,a,p)&\Im f'_3(y,a,p)&\Im f'_4(y,a,p)
\end{pmatrix}.$$
\begin{rem}
First, note that in order for $F(\cdot,a,p)$ to be a fundamental system, we need to ensure linear independence of $f_j$. Moreover, the $\mathbb{R}^{4\times 4}$-matrix $M(a,p)$ transforms the system on the left-hand side,  to match the right-hand side system at $0$. This yields the global system $F$, such that $f_j(\cdot,a,p)\in C^1(-1,1)$. However, this transformation makes the singular behavior of $f_{L,3}$ and $f_{L,4}$ contribute to all four fundamental solutions $f_j$. By requiring $M(a,p)_1^3\neq 0$, we can ensure that $f_1$ is singular at $-1$, which will allow us to cancel one of the singular terms (see Definition \ref{LinvDef}).
\end{rem}
\begin{lem}\label{admfundmlem}
    There exists an open interval $\mathring{I}\subset(\tfrac{7}{3},5)$ such that $3\in \mathring{I}$, $\det(F_L(y))\neq 0$ for all $y\in(-1,0)$, $\det(F_R(y,a,p))\neq 0$ for all $(y,a,p)\in(0,1)\times J_*\times I$ and $M(a,p)_1^3\neq 0$ for all $(a,p)\in J_*\times I$, where $I$ denotes the closure of $\mathring{I}$.
\end{lem}
\begin{proof}
    By \cite{DS}, Lemma 6.8, $\det(F_L(y))\neq 0$ for all $y\in(-1,0)$. Analogously to Definition 6.14 in \cite{DS}, we factor out the oscillation and the singularity at $1$, by decomposing
    $$F_R(y,a,p)=\frac{P_0(y,a,p)}{a_*+a}+(1-y)^{\tfrac{4}{p-1}-4}\frac{\cos(\varphi(y,a_*+a))}{(a_*+a)^2}P_C(y,a)+(1-y)^{\tfrac{4}{p-1}-4}\frac{\sin(\varphi(y,a_*+a))}{(a_*+a)^2}P_S(y,a),$$
    where $P_C,P_S$ are given in Definition 6.14 in \cite{DS} and they are of the form
    $$\begin{pmatrix}
        0&0&(P_X)^1_3&(P_X)^1_4\\
        0&0&(P_X)^2_3&(P_X)^2_4\\
        0&0&(P_X)^3_3&(P_X)^3_4\\
        0&0&(P_X)^4_3&(P_X)^4_4
    \end{pmatrix},$$
    with polynomial entries $(P_X)^j_k$ for $X\in\lbrace C,S\rbrace$. We set
    $$P_0(y,a,p):=\begin{pmatrix}(a_*+a)\Re f_{R,1}(y,a,p)&(a_*+a)\Re f_{R,2}(y,a,p)&0&0\\
(a_*+a)\Im f_{R,1}(y,a,p)&(a_*+a)\Im f_{R,2}(y,a,p)&0&0\\
(a_*+a)\Re f'_{R,1}(y,a,p)&(a_*+a)\Re f'_{R,2}(y,a,p)&0&0\\
(a_*+a)\Im f'_{R,1}(y,a,p)&(a_*+a)\Im f'_{R,2}(y,a,p)&0&0
\end{pmatrix},$$
such that by multilinearity
\begin{align*}
    (a_*+a)^6(1-y)^{2\left(4-\tfrac{4}{p-1}\right)}\det(F_R(y,a,p))=&\cos(\varphi(y,a_*+a))^2P_{d,R}(y,a,p)\\&+\cos(\varphi(y,a_*+a))\sin(\varphi(y,a_*+a))Q_2(y,a,p)\\&+\sin(\varphi(y,a_*+a))^2Q_3(y,a,p)\\=&P_{d,R}(y,a,p)+\cos(\varphi(y,a_*+a))\sin(\varphi(y,a_*+a))Q_2(y,a,p)\\&+\sin(\varphi(y,a_*+a))^2\left[Q_3(y,a,p)-P_{d,R}(y,a,p)\right],
\end{align*}
where
$$P_{d,R}=\det(P_0+P_C),\hspace{0.5cm}Q_2=\frac{\det(P_0+P_C+P_S)-\det(P_0+P_C-P_S)}{2},\hspace{0.5cm}Q_3=\det(P_0+P_S).$$
Since all functions that appear in $P_0$, $P_C$, and $P_S$ are polynomial in $y$, the functions $P_{d,R}$, $Q_2$, and $Q_3$ extend continuously to $[0,1]\times J_*\times I$. By Lemma 6.19 in \cite{DS}, $$\min_{y\in[0,1],a\in J_*}P_{d,R}(y,a,3)\geq 8,$$ $$\sup_{y\in[0,1],a\in J_*}|Q_2(y,a,3)|\leq 1,\hspace{1cm}\sup_{y\in[0,1],a\in J_*}|Q_3(y,a,3)-P_{d,R}(y,a,3)|\leq 1,$$
and hence, by continuity, $\det(F_R(y,a,p))\neq 0$ for all $(y,a,p)\in(0,1)\times J_*\times I$, provided $I$ is chosen sufficiently small.

From the proof of Lemma 6.22 in \cite{DS} we have $M(a,3)_1^3=\left[F_L^{-1}(0)F_R(0,a,3)\right]_1^3\leq -\frac{1}{5}$, and hence, by the continuity of $(a,p)\mapsto F_R(0,a,p):J_*\times I\rightarrow\mathbb{R}$, $M(a,p)_1^3< 0$ for all $(a,p)\in J_*\times I$ with $I$ sufficiently small around $3$.
\end{proof}
\noindent From now on, let $I\subset(\tfrac{7}{3},5)$ be a compact interval as in Lemma \ref{admfundmlem}.
\begin{lem}\label{asymFRinvlem}
Let $(a,p)\in J_*\times I$. Then there exists a $\delta\in(0,1]$ such that
$$F^{-1}(y,a,p)=\begin{pmatrix}
    \phi\left(\frac{f_{3}'(y,a,p)}{W(y,a,p)}\right)&-\phi\left(\frac{f_{3}(y,a,p)}{W(y,a,p)}\right)\\-\phi\left(\frac{f_{1}'(y,a,p)}{W(y,a,p)}\right)&\phi\left(\frac{f_{1}(y,a,p)}{W(y,a,p)}\right)
\end{pmatrix}\left[1+O((1-y)^2a^0p^0)\right]$$
for all $(y,a,p)\in [1-\delta,1)\times J_*\times I$, where
\begin{align*}
    W(y,a,p):&=f_{1}(y,a,p)f_{3}'(y,a,p)-f_{1}'(y,a,p)f_{3}(y,a,p)\\
    &=-4i(a_*+a)e^{i\varphi(y,a_*+a)}(1-y)^{-\tfrac{4}{p-1}}\left[1+\left(1+\tfrac{3-p}{2(p-1)}\right)(1-y)+O((1-y)^2a^0p^0)\right]
\end{align*}
and $\phi:\mathbb{C}\rightarrow\mathbb{R}^{2\times 2}$ is defined by
$$\phi(z):=\begin{pmatrix}\Re z&-\Im z\\
\Im z&\Re z
\end{pmatrix}.$$
\end{lem}
\begin{proof}
    Let $y\in [0,1)$ and $(a,p)\in J_*\times I$. For the Wronskian note that
    $f_1'(y,a,p)=O(y^0a^0p^0)$ and
    \begin{align*}
        f_3'(y,a,p)=&i\varphi'(y,a_*+a)e^{i\varphi(y,a_*+a)}(1-y)^{3-\tfrac{4}{p-1}}\left[1+\frac{2(a_*+a)-i}{2(a_*+a)}(1-y)+O((1-y)^2a^0p^0)\right]\\
        &+O((1-y)^{2-\tfrac{4}{p-1}}a^0p^0).
    \end{align*}
    Thus, since
    $$\varphi'(y,a_*+a)=-\frac{4(a_*+a)}{(1-y)^3}\left[1-\tfrac{1}{2}(1-y)+O((1-y)^2a^0)\right],$$
    we obtain
    \begin{align*}
        f_3'(y,a,p)=&-4i(a_*+a)e^{i\varphi(y,a_*+a)}(1-y)^{-\tfrac{4}{p-1}}\left[1+\frac{2(a_*+a)-i}{2(a_*+a)}(1-y)+O((1-y)^2a^0p^0)\right]\\
        &\times \left[1-\tfrac{1}{2}(1-y)+O((1-y)^2a^0)\right]+O((1-y)^{2-\tfrac{4}{p-1}}a^0p^0)\\
        =&-4i(a_*+a)e^{i\varphi(y,a_*+a)}(1-y)^{-\tfrac{4}{p-1}}\left[1+\frac{a_*+a-i}{2(a_*+a)}(1-y)+O((1-y)^2a^0p^0)\right].
    \end{align*}
    Consequently,
    \begin{align*}
        W(y,a,p)=&f_1(y,a,p)f_3'(y,a,p)+O\left((1-y)^{3-\tfrac{4}{p-1}}a^0p^0\right)\\
        =&-4i(a_*+a)e^{i\varphi(y,a_*+a)}(1-y)^{-\tfrac{4}{p-1}}\left[1+\frac{a_*+a-i}{2(a_*+a)}(1-y)+O((1-y)^2a^0p^0)\right]\\
        &\times\left[1+\left(\frac{a_*+a+i}{2(a_*+a)}+\frac{3-p}{2(p-1)}\right)(1-y)+O((1-y)^2)\right]\\
        =&-4i(a_*+a)e^{i\varphi(y,a_*+a)}(1-y)^{-\tfrac{4}{p-1}}\left[1+\left(1+\tfrac{3-p}{2(p-1)}\right)(1-y)+O((1-y)^2a^0p^0)\right].
    \end{align*}
    Moreover, it follows that there exists a $\delta\in(0,1]$ such that $W(y,a,p)\neq0$ for all $(y,a,p)\in[1-\delta,1)\times J_*\times I$.
    
    Now observe that
    \begin{align*}
        f_2(y,a,p)=&if_1(y,a,p)+O((1-y)^2),\\
        f_2'(y,a,p)=&if_1'(y,a,p)+O((1-y)),\\
        f_4(y,a,p)=&if_3(y,a,p)+O\left((1-y)^{5-\tfrac{4}{p-1}}a^0p^0\right),\\
        f_4'(y,a,p)=&if_3'(y,a,p)+O\left((1-y)^{2-\tfrac{4}{p-1}}a^0p^0\right)
    \end{align*}
    and thus,
    $$F(y,a,p)=\begin{pmatrix}
        \phi(f_1(y,a,p))&\phi(f_3(y,a,p))\\\phi(f_1'(y,a,p))&\phi(f_3'(y,a,p))
    \end{pmatrix}+\begin{pmatrix}
        0&O((1-y)^2)&0&O\left((1-y)^{5-\tfrac{4}{p-1}}a^0p^0\right)\\
        0&O((1-y)^2)&0&O\left((1-y)^{5-\tfrac{4}{p-1}}a^0p^0\right)\\
        0&O((1-y))&0&O\left((1-y)^{2-\tfrac{4}{p-1}}a^0p^0\right)\\
        0&O((1-y))&0&O\left((1-y)^{2-\tfrac{4}{p-1}}a^0p^0\right)
    \end{pmatrix}.$$
    Since $\phi(z_1z_2)=\phi(z_1)\phi(z_2)$ for all $z_1,z_2\in\mathbb{C}$, the first matrix can be formally inverted like a $2\times 2$-matrix and we obtain
    \begin{align*}
        F(y,a,p)=&\begin{pmatrix}
        \phi(f_1(y,a,p))&\phi(f_3(y,a,p))\\\phi(f_1'(y,a,p))&\phi(f_3'(y,a,p))
    \end{pmatrix}\\
    &\times\left[1+\begin{pmatrix}
        \phi\left(\frac{f_3'(y,a,p)}{W(y,a,p)}\right)&-\phi\left(\frac{f_3(y,a,p)}{W(y,a,p)}\right)\\-\phi\left(\frac{f_1'(y,a,p)}{W(y,a,p)}\right)&\phi\left(\frac{f_1(y,a,p)}{W(y,a,p)}\right)
    \end{pmatrix}\begin{pmatrix}
        0&O((1-y)^2)&0&O\left((1-y)^{5-\tfrac{4}{p-1}}a^0p^0\right)\\
        0&O((1-y)^2)&0&O\left((1-y)^{5-\tfrac{4}{p-1}}a^0p^0\right)\\
        0&O((1-y))&0&O\left((1-y)^{2-\tfrac{4}{p-1}}a^0p^0\right)\\
        0&O((1-y))&0&O\left((1-y)^{2-\tfrac{4}{p-1}}a^0p^0\right)
    \end{pmatrix}\right]\\
    =&\begin{pmatrix}
        \phi(f_1(y,a,p))&\phi(f_3(y,a,p))\\\phi(f_1'(y,a,p))&\phi(f_3'(y,a,p))
    \end{pmatrix}\left[1+O((1-y)^2a^0p^0)\right]
    \end{align*}
    for all $(y,a,p)\in[1-\delta,1)\times J_*\times I$, which yields the stated form of $F^{-1}$.
\end{proof}
\begin{defin}\label{coeffdef}
    The functions $(p_1,p_2,q_1,q_2)$ defined on $(-1,1)\setminus\lbrace 0\rbrace\times J_*\times I$ by
        \begin{align*}
        p_1&=\frac{A_3^3+A_4^4}{2}+i\frac{A_3^4-A_4^3}{2},&&p_2=\frac{A_3^3-A_4^4}{2}+i\frac{A^4_3+A_4^3}{2},\\
        q_1&=\frac{A^3_1+A^4_2}{2}+i\frac{A_1^4-A_2^3}{2},&&q_2=\frac{A_1^3-A_2^4}{2}+i\frac{A^4_1+A^3_2}{2},
    \end{align*}
    where $A(y,a,p):=-F'(y,a,p)F^{-1}(y,a,p)$, are called the \textit{coefficients associated to $F$}.
\end{defin}

\begin{prop}\label{coeffFprop}
    We have $f_j(\cdot,a,p)\in C^2((-1,1)\setminus\lbrace 0\rbrace)\cap C^1(-1,1)$ for all $(a,p)\in J_*\times I$, $j\in\lbrace1,2,3,4\rbrace$ and $$f_j''(y,a,p)+p_1(y,a,p)f_j'(y,a,p)+p_2(y,a,p)\overline{f_j'(y,a,p)}+q_1(y,a,p)f_j(y,a,p)+q_2(y,a,p)\overline{f_j(y,a,p)}=0$$
    for all $j\in\lbrace 1,2,3,4\rbrace$, $(y,a,p)\in(-1,1)\setminus\lbrace 0\rbrace\times J_*\times I$, where $(p_1,p_2,q_1,q_2)$ are the coefficients associated to $F$. Furthermore,
    \begin{align*}
        p_1(y,a,p)&=\frac{4i(a_*+a)}{(1-y)^3}-\frac{2i(a_*+a)}{(1-y)^2}+\frac{2}{1+y}+O((1-y)^{-1}a^0p^0),\\
        p_2(y,a,p)&=O((1-y)^{-1}a^0p^0),\\
        q_1(y,a,p)&=\frac{\frac{4i(a_*+a)}{p-1}-2}{(1-y)^3}+O((1-y)^{-2}a^0p^0)+O((1+y)^{-1}),\\
        q_2(y,a,p)&=O((1-y)^{-1}a^0p^0)+O((1+y)^{-1}).
    \end{align*}
\end{prop}
\begin{proof}
Let $(a,p)\in J_*\times I$. We have $F_L\in C^1((-1,0),\mathbb{R}^{4\times 4})$, $F_R(\cdot,a,p)\in C^1((0,1),\mathbb{R}^{4\times 4})$ and consequently, $f_j(\cdot,a,p)\in C^2((-1,1)\setminus\lbrace 0\rbrace)\cap C^1(-1,1)$ for all $j\in\lbrace 1,2,3,4\rbrace$. By Definition \ref{coeffdef}, we have
\begin{equation}\label{FAeq}
F'(y,a,p)+A(y,a,p)F(y,a,p)=0    
\end{equation}
\noindent and since $F'(y,a,p)^j=F(y,a,p)^{j+2}$ for $j\in\lbrace 1,2\rbrace$, it follows that
$$A(y,a,p)_k^j=-\sum_{l=1}^4F'(y,a,p)_l^jF^{-1}(y,a,p)_k^l=-\sum_{l=1}^4F(y,a,p)_l^{j+2}F^{-1}(y,a,p)_k^l=-\delta^{j+2}_k$$
for $j\in\lbrace 1,2\rbrace$. Thus, $A$ is of the form
$$A=\begin{pmatrix}
    0&0&-1&0\\
    0&0&0&-1\\
    A^3_1&A^3_2&A^3_3&A^3_4\\
    A^4_1&A^4_2&A^4_3&A^4_4
\end{pmatrix}.$$
Consequently, Eq.~(\ref{FAeq}) is equivalent to
$$f_j''(y,a,p)+p_1(y,a,p)f_j'(y,a,p)+p_2(y,a,p)\overline{f_j'(y,a,p)}+q_1(y,a,p)f_j(y,a,p)+q_2(y,a,p)\overline{f_j(y,a,p)}=0$$
    for all $j\in\lbrace 1,2,3,4\rbrace$.
    
    In the case $y\in (0,1)$, we have by definition
    \begin{align*}
        f_2(y,a,p)=&if_1(y,a,p)+O((1-y)^2),\\
        f_2'(y,a,p)=&if_1'(y,a,p)+O((1-y)),\\
        f_2''(y,a,p)=&if_1''(y,a,p)+O(y^0),
    \end{align*}
    as well as
    \begin{align*}
        f_4(y,a,p)=&if_3(y,a,p)+O\left((1-y)^{5-\tfrac{4}{p-1}}a^0p^0\right),\\
        f_4'(y,a,p)=&if_3'(y,a,p)+O\left((1-y)^{2-\tfrac{4}{p-1}}a^0p^0\right),\\
        f_4''(y,a,p=&if_3''(y,a,p)+O\left((1-y)^{-1-\tfrac{4}{p-1}}a^0p^0\right).
    \end{align*}
    Consequently,
    $$F'(y,a,p)=\begin{pmatrix}
        \phi(f_1'(y,a,p))&\phi(f_3'(y,a,p))\\
        \phi(f_1''(y,a,p))&\phi(f_3''(y,a,p))
    \end{pmatrix}+\begin{pmatrix}
        0&O((1-y))&0&O\left((1-y)^{2-\tfrac{4}{p-1}}a^0p^0\right)\\
        0&O((1-y))&0&O\left((1-y)^{2-\tfrac{4}{p-1}}a^0p^0\right)\\
        0&O(y^0)&0&O\left((1-y)^{-1-\tfrac{4}{p-1}}a^0p^0\right)\\
        0&O(y^0)&0&O\left((1-y)^{-1-\tfrac{4}{p-1}}a^0p^0\right)
    \end{pmatrix}$$
    and by Lemma \ref{asymFRinvlem},
    \begin{align*}
        A(y,a,p)=&-F'(y,a,p)F^{-1}(y,a,p)\\
        =&F'(y,a,p)\begin{pmatrix}
    -\phi\left(\frac{f_{3}'(y,a,p)}{W(y,a,p)}\right)&\phi\left(\frac{f_{3}(y,a,p)}{W(y,a,p)}\right)\\ \phi\left(\frac{f_{1}'(y,a,p)}{W(y,a,p)}\right)&-\phi\left(\frac{f_{1}(y,a,p)}{W(y,a,p)}\right)
\end{pmatrix}\left[1+O((1-y)^2a^0p^0)\right]\\
=&\begin{pmatrix}
    \phi(0)&-\phi(1)\\
    \phi\left(\frac{f_1'(y,a,p)f_3''(y,a,p)-f_1''(y,a,p)f_3'(y,a,p)}{W(y,a,p)}\right)&-\phi\left(\frac{f_1(y,a,p)f_3''(y,a,p)-f_1''(y,a,p)f_3(y,a,p)}{W(y,a,p)}\right)
\end{pmatrix}\\ &\times\left[1+O((1-y)^2a^0p^0)\right]\\
&+\begin{pmatrix}
    O((1-y)a^0p^0)&O((1-y)a^0p^0)&O((1-y)^2a^0p^0)&O((1-y)^2a^0p^0)\\
    O((1-y)a^0p^0)&O((1-y)a^0p^0)&O((1-y)^2a^0p^0)&O((1-y)^2a^0p^0)\\
    O((1-y)^{-1}a^0p^0)&O((1-y)^{-1}a^0p^0)&O((1-y)^{-1}a^0p^0)&O((1-y)^{-1}a^0p^0)\\
    O((1-y)^{-1}a^0p^0)&O((1-y)^{-1}a^0p^0)&O((1-y)^{-1}a^0p^0)&O((1-y)^{-1}a^0p^0)
\end{pmatrix}.
    \end{align*}
In other words,
\begin{align*}
        A(y,a,p)=&\begin{pmatrix}
    0&0&-1&0\\0&0&0&-1\\
    \textrm{Re }h_1(y,a,p)&-\textrm{Im }h_1(y,a,p)&\textrm{Re }h_2(y,a,p)&-\textrm{Im }h_2(y,a,p)\\
    \textrm{Im }h_1(y,a,p)&\textrm{Re }h_1(y,a,p)&\textrm{Im }h_2(y,a,p)&\textrm{Re }h_2(y,a,p)
\end{pmatrix}\left[1+O((1-y)^2a^0p^0)\right]\\
&+\begin{pmatrix}
    0&0&0&0\\
    0&0&0&0\\
    O((1-y)^{-1}a^0p^0)&O((1-y)^{-1}a^0p^0)&O((1-y)^{-1}a^0p^0)&O((1-y)^{-1}a^0p^0)\\
    O((1-y)^{-1}a^0p^0)&O((1-y)^{-1}a^0p^0)&O((1-y)^{-1}a^0p^0)&O((1-y)^{-1}a^0p^0)
\end{pmatrix},
    \end{align*}
where
    \begin{align*}
        h_1(y,a,p)&:=\frac{f_1'(y,a,p)f_3''(y,a,p)-f_1''(y,a,p)f_3'(y,a,p)}{W(y,a,p)},\\
        h_2(y,a,p)&:=-\frac{f_1(y,a,p)f_3''(y,a,p)-f_1''(y,a,p)f_3(y,a,p)}{W(y,a,p)},
    \end{align*}
and hence,
    \begin{align*}
        p_1(y,a,p)&=h_2(y,a_*+a,p)\left[1+O((1-y)^2a^0p^0)\right]+O((1-y)^{-1}a^0p^0),\\
        p_2(y,a,p)&=h_2(y,a_*+a,p)O((1-y)^2a^0p^0)+O((1-y)^{-1}a^0p^0),\\
        q_1(y,a,p)&=h_1(y,a_*+a,p)\left[1+O((1-y)^2a^0p^0)\right]+O((1-y)^{-1}a^0p^0),\\
        q_2(y,a,p)&=h_1(y,a_*+a,p)O((1-y)^2a^0p^0)+O((1-y)^{-1}a^0p^0).
    \end{align*}
From Lemma \ref{asymFRinvlem} recall that
    $$W(y,a,p)=-4i(a_*+a)e^{i\varphi(y,a_*+a)}(1-y)^{-\tfrac{4}{p-1}}\left[1+\left(1+\tfrac{3-p}{2(p-1)}\right)(1-y)+O((1-y)^2a^0p^0)\right]$$
    and by definition we have $f_3'(y,a,p)=O((1-y)^{-\frac{4}{p-1}}a^0p^0)$ and
\begin{align*}
    f_3''(y,a,p)=&-\varphi'(y,a_*+a)^2e^{i\varphi(y,a_*+a)}(1-y)^{3-\tfrac{4}{p-1}}\left[1+\frac{2(a_*+a)-i}{2(a_*+a)}(1-y)+O((1-y)^{2}a^0p^0)\right]\\&+O\left((1-y)^{-1-\tfrac{4}{p-1}}a^0p^0\right)\\
    =&-16(a_*+a)^2e^{i\varphi(y,a_*+a)}(1-y)^{-3-\tfrac{4}{p-1}}\left[1-\frac{i}{2(a_*+a)}(1-y)+O((1-y)^2a^0p^0)\right],
\end{align*}
which yields
$$\frac{f_3''(y,a,p)}{W(y,a,p)}=-\frac{4i(a_*+a)}{(1-y)^3}\left[1-\left(\frac{2(a_*+a)+i}{2(a_*+a)}+\frac{3-p}{2(p-1)}\right)(1-y)+O((1-y)^2a^0p^0)\right].$$
Recalling that $f_1(y,a,p)=1+\left[\frac{a_*+a+i}{2(a_*+a)}+\frac{3-p}{2(p-1)}\right](1-y)+O((1-y)^2)$, we obtain
\begin{align*}
    h_1(y,a,p)&=f_1'(y,a,p)\frac{f_3''(y,a,p)}{W(y,a,p)}+O((1-y)^0a^0p^0)\\
    &=\left[-\left(\frac{a_*+a+i}{2(a_*+a)}+\frac{3-p}{2(p-1)}\right)+O(1-y)\right]\frac{f_3''(y,a,p)}{W(y,a,p)}+O((1-y)^0a^0p^0)\\
    &=\frac{\frac{4i(a_*+a)}{p-1}-2}{(1-y)^3}\left[1+O((1-y)a^0p^0)\right]
\end{align*}
    and
    \begin{align*}
        h_2(y,a,p)=&-f_1(y,a,p)\frac{f_3''(y,a,p)}{W(y,a,p)}+O((1-y)^3a^0p^0)\\
        =&\frac{4i(a_*+a)}{(1-y)^3}\left[1+\left(\frac{a_*+a+i}{2(a_*+a)}+\frac{3-p}{2(p-1)}\right)(1-y)+O((1-y)^2)\right]\\&\times\left[1-\left(\frac{2(a_*+a)+i}{2(a_*+a)}+\frac{3-p}{2(p-1)}\right)(1-y)+O((1-y)^2a^0p^0)\right]\\
        =&\frac{4i(a_*+a)}{(1-y)^3}\left[1-\frac{1}{2}(1-y)+O((1-y)^2a^0p^0)\right].
    \end{align*}
    This yields the claimed asymptotics for the coefficients on the right-hand side.
    
    In the case $y\in(-1,0)$, we have 
\begin{align*}
    A(y,a,p)=&-F'(y,a,p)F^{-1}(y,a,p)=-F_L'(y)F_L^{-1}(0)F_R(0,a,p)F_R^{-1}(0,a,p)F_L(0)F_L^{-1}(y)\\
    =&-F_L'(y)F_L^{-1}(y).
\end{align*}
In particular, the asymptotics for the coefficients on the left-hand side are independent of $a$ and $p$ and the proof of the stated asymptotics can be found in \cite{DS}, Proposition 3.21.
 \end{proof}
 
 \begin{defin}
     Let $\mathcal{I}\subset(-1,1)$ be open. Then, for $f\in C^2(\mathcal{I})$ and $(y,a,p)\in \mathcal{I}\times J_*\times I$, we set
     $$\tilde{\mathcal{L}}_{a,p}(f)(y):=f''(y)+p_1(y,a,p)f'(y)+p_2(y,a,p)\overline{f'(y)}+q_1(y,a,p)f(y)+q_2(y,a,p)\overline{f(y)}.$$
 \end{defin}
 \subsection{The inhomogeneous equation}
 \begin{defin}
     For $(y,a,p)\in(-1,1)\times J_*\times I$, $k\in\lbrace 1,2,3,4\rbrace$, and $g\in Y$, we set
     $$\alpha^k(a,p,g)(y):=F^{-1}(y,a,p)^k_3\textnormal{ Re }g(y)+F^{-1}(y,a,p)^k_4\textnormal{ Im }g(y).$$
 \end{defin}
 \begin{lem}\label{alphabdlem}
We have the bound
$$\left|\alpha^k(a,p,g)(y)\right|\lesssim\|g\|_Y$$
for $k\in\lbrace 1,2\rbrace$, and
$$\left|\alpha^k(a,p,g)(y)\right|\lesssim(1-y)^{\tfrac{4}{p-1}-2}\|g\|_Y$$
for $k\in\lbrace 3,4\rbrace$, for all $(a,p,g)\in J_*\times I\times Y$. In particular,
$$\int_\mathcal{I}\left|\alpha^k(a,p,g)(y)\right|dy\lesssim\|g\|_Y$$
for all $k\in\lbrace 1,2,3,4\rbrace$, $\mathcal{I}\subseteq(-1,1)$ and $(a,p,g)\in J_*\times I\times Y$.
 \end{lem}
 \begin{proof}
 From Lemma \ref{asymFRinvlem} we have
 $$\left|F^{-1}(y,a,p)^k_3\right|+\left|F^{-1}(y,a,p)^k_4\right|\lesssim \begin{cases}
  (1-y)^2,  & k\in\lbrace 1,2\rbrace \\
  (1-y)^{\tfrac{4}{p-1}} ,  & k\in\lbrace 3,4\rbrace
\end{cases} $$
for all $(y,a,p)\in [0,1)\times J_*\times I$. Furthermore, if $y\in (-1,0]$,
$$F^{-1}(y,a,p)_l^k=\sum_{m=1}^4M^{-1}(a,p)_m^kF_L^{-1}(y)_l^m,$$
and from \cite{DS}, Lemma 3.16,
$$F_L^{-1}(y)=\begin{pmatrix}
    1&O(1+y)&1+y&O((1+y)^2)\\
    O(1+y)&1&O((1+y)^2)&1+y\\
    O((1+y)^2)&O((1+y)^2)&-(1+y)^2&O((1+y)^3)\\
    O((1+y)^2)&O((1+y)^2)&O((1+y)^3)&-(1+y)^2
\end{pmatrix}\left[1+O(1+y)\right],$$
which yields
     $$\left|F^{-1}(y,a,p)^k_3\right|+\left|F^{-1}(y,a,p)^k_4\right|\lesssim 1+y$$
     for all $(y,a,p)\in (-1,0]\times J_*\times I$.\\
 \end{proof}
\begin{defin}\label{LinvDef}
    For $(y,a,p)\in (-1,1)\times J_*\times I$ and $g\in Y$ we set
    \begin{align*}
        \tilde{\mathcal{L}}^{-1}_{a,p}(g)(y):=&\sum_{k=1}^2\left[F(y,a,p)_k^1+iF(y,a,p)_k^2\right]\int_{-1}^y\alpha^k(a,p,g)(x)dx\\
        &-\sum_{k=3}^4\left[F(y,a,p)_k^1+iF(y,a,p)_k^2\right]\int^{1}_y\alpha^k(a,p,g)(x)dx\\
        &+\left[F(y,a,p)_1^1+iF(y,a,p)_1^2\right]\sum_{k=3}^4\frac{M(a,p)_k^3}{M(a,p)_1^3}\int_{-1}^1\alpha^k(a,p,g)(x)dx.
    \end{align*}
\end{defin}
\begin{rem}
    The first two terms in the definition of $\tilde{\mathcal{L}}^{-1}_{a,p}$ are straightforward from the variation of constants formula. The last term cancels the contribution of the singularity at $-1$ from $f_{L,3}$ to $f_3$ and $f_4$. See \cite{DS}, Lemma 3.27, for the proof of the following Lemma.
\end{rem}
\begin{lem}\label{Linvlem}
    Let $g\in Y$. Then $\tilde{\mathcal{L}}^{-1}_{a,p}(g)\in C^2((-1,1)\setminus\lbrace 0\rbrace)\cap C^1(-1,1)$ and
    $$\tilde{\mathcal{L}}_{a,p}\left(\tilde{\mathcal{L}}^{-1}_{a,p}(g)\right)(y)=g(y)$$
    for all $(y,a,p)\in (-1,1)\setminus\lbrace 0\rbrace\times J_*\times I$.
\end{lem}
\noindent The problem with the inverse $\tilde{\mathcal{L}}^{-1}_{a,p}$ is that is does not map from $Y$ to $X$ because of the singularity at $-1$, which $f_3$ and $f_4$ inherit from $f_{L,4}$. We will correct this by adding a regularization term that subtracts the right amount of the singular function $f_{L,4}$ near $-1$.
\begin{defin}
We define $\psi: J_*\times I\times Y\rightarrow \mathbb{R}$ by
$$\psi(a,p,g):=\sum_{k=3}^4\left[M(a,p)_1^4\frac{M(a,p)_k^3}{M(a,p)_1^3}-M(a,p)_k^4\right]\int_{-1}^1\alpha^k(a,p,g)(x)dx.$$
\end{defin}
\begin{lem}\label{psicontlem}
    We have the bound
    $$|\psi(a,p,g)|\lesssim\|g\|_Y$$
    for all $g\in Y$ and $(a,p)\in J_*\times I$. Furthermore, for every $g\in Y$, the map
    $$(a,p)\mapsto\psi(a,p,g):J_*\times I\rightarrow\mathbb{R}$$
    is continuous.
\end{lem}
\begin{proof}
    The bound follows directly from Lemma \ref{alphabdlem}. For the continuity statement observe that by Lemma \ref{alphabdlem}, for given $\varepsilon>0$, there exist $\eta,\delta>0$ such that
    $$\int_{1-\eta}^1\left|\alpha^k(a,p,g)-\alpha^k(b,q,g)\right|<\frac{\varepsilon}{2}$$
    for all $(a,p),(b,q)\in J_*\times I$.
    Furthermore, for all $g\in Y$, $k\in\lbrace 1,2,3,4\rbrace$, the map $(y,a,p)\mapsto\alpha^k(a,p,g)(y)$ is continuous on $[-1,1-\eta]\times J_*\times I$, which gives
    $$\int^{1-\eta}_{-1}\left|\alpha^k(a,p,g)-\alpha^k(b,q,g)\right|<\frac{\varepsilon}{2},$$
    provided $|a-b|+|p-q|<\delta$.
\end{proof}
 
\begin{defin}
    Let $(a,p,g)\in J_*\times I\times Y$ and let $\chi:\mathbb{R}\rightarrow[0,1]$ be a smooth cut-off that satisfies $\chi(y)=1$ for $y\leq-\frac{1}{2}$, $\chi(y)=0$ for $y\geq 0$, and $|\chi'(y)|\leq3$ for all $y\in\mathbb{R}$. Then we set
    $$\mathcal{J}(a,p,g)(y):=\tilde{\mathcal{L}}^{-1}_{a,p}(g)(y)-\chi(y)\left[F_L(y)_4^1+iF_L(y)_4^2\right]\psi(a,p,g)$$
    for $y\in(-1,1)$.
\end{defin}
\begin{rem}
    Note that $\mathcal{J}(a,p,\cdot)$ is not an inverse of $\tilde{\mathcal{L}}_{a,p}$, unless $\psi(a,p,g)=0$.
\end{rem}

\begin{prop}\label{Jbdlem}
    We have the bound
    $$\|\mathcal{J}(a,p,g)\|_X\lesssim\|g\|_Y$$
    for all $(a,p,g)\in J_*\times I\times Y$.
\end{prop}
\begin{proof}
For $(y,a,p)\in[0,1)\times J_*\times I$ we have $\mathcal{J}(a,p,g)=\tilde{\mathcal{L}}^{-1}_{a,p}(g)$ and
$$|F(y,a,p)^1_k+iF(y,a,p)_k^2|\lesssim 1,\hspace{1cm}|F(y,a,p)_k^3+iF(y,a,p)_k^4|\lesssim 1$$
for $k\in\lbrace 1,2\rbrace$ and
$$|F(y,a,p)^1_k+iF(y,a,p)^2_k|\lesssim (1-y)^{3-\tfrac{4}{p-1}}\lesssim1,\hspace{1cm}|F(y,a,p)^3_k+iF(y,a,p)^4_k|\lesssim (1-y)^{-\tfrac{4}{p-1}}$$

\noindent for $k\in\lbrace 3,4\rbrace$. Hence, from Lemma \ref{alphabdlem} we obtain
$$\left|\mathcal{J}(a,p,g)(y)\right|\lesssim \|g\|_Y$$
and
    \begin{align*}
    \left|\mathcal{J}(a,p,g)'(y)\right|=
        &\Bigg|\sum_{k=1}^2\left[F(y,a,p)_k^3+iF(y,a,p)_k^4\right]\int_{-1}^y\alpha^k(a,p,g)(x)dx\\&+\sum_{k=1}^2\left[F(y,a,p)_k^1+iF(y,a,p)_k^2\right]\alpha^k(a,p,g)(y)\\
        &-\sum_{k=3}^4\left[F(y,a,p)_k^3+iF(y,a,p)_k^4\right]\int^{1}_y\alpha^k(a,p,g)(x)dx\\&+\sum_{k=3}^4\left[F(y,a,p)_k^1+iF(y,a,p)_k^2\right]\alpha^k(a,p,g)(y)\\
        &+\left[F(y,a,p)_1^3+iF(y,a,p)_1^4\right]\sum_{k=3}^4\frac{M(a,p)_k^3}{M(a,p)_1^3}\int_{-1}^1\alpha^k(a,p,g)(x)dx\Bigg|\\
        \lesssim&\ \|g\|_Y+(1-y)^{-\tfrac{4}{p-1}}\int_y^1(1-x)^{\tfrac{4}{p-1}-2}\|g\|_Ydx+(1-y)^{3-\tfrac{4}{p-1}}(1-y)^{\tfrac{4}{p-1}-2}\|g\|_Y\\
        \lesssim&(1-y)^{-1}\|g\|_Y
    \end{align*}for all $(y,a,p)\in[0,1)\times J_*\times I$.
    
    In the case $y\in(-1,0]$, we have $F(y,a,p)=F_L(y)M(a,p)$. Consequently, if $y\in(-1,-\tfrac{1}{2}]$,
    \begin{align*}
        \mathcal{J}(a,p,g)(y)=&\tilde{\mathcal{L}}_{a,p}^{-1}(a,p,g)(y)-\left[F_L(y)^1_4+iF_L(y)^2_4\right]\psi(a,p,g)\\
        =&\sum_{k=1}^2\sum_{l=1}^4\left[F_L(y)_l^1+iF_L(y)_l^2\right]M(a,p)_k^l\int_{-1}^y\alpha^k(a,p,g)\\
        &-\sum_{k=3}^4\sum_{l=1}^4\left[F_L(y)_l^1+iF_L(y)_l^2\right]M(a,p)_k^l\int_y^1\alpha^k(a,p,g)\\
        &+\sum_{k=3}^4\sum_{l=1}^4\left[F_L(y)_l^1+iF_L(y)_l^2\right]M(a,p)_1^l\frac{M(a,p)_k^3}{M(a,p)_1^3}\int_{-1}^1\alpha^k(a,p,g)
        \\&-\sum_{k=3}^4\left[M(a,p)_1^4\frac{M(a,p)^3_k}{M(a,p)_1^3}-M(a,p)_k^4\right]\left[F_L(y)_4^1+iF_L(y)_4^2\right]\int_{-1}^1\alpha^k(a,p,g).
    \end{align*}
    Since $$\left|F_L(y)_l^1+iF_L(y)_l^2\right|\lesssim1,\hspace{1cm}\left|F_L(y)_l^3+iF_L(y)_l^4\right|\lesssim1$$
    for $l\in\lbrace 1,2\rbrace$ and
    $$\left|F_L(y)_l^1+iF_L(y)_l^2\right|\lesssim(1+y)^{-1},\hspace{1cm}\left|F_L(y)_l^3+iF_L(y)_l^4\right|\lesssim(1+y)^{-2}$$
    for $l\in\lbrace 3,4\rbrace$, we obtain from Lemma \ref{alphabdlem}
    
    \begin{align*}
        \mathcal{J}(a,p,g)(y)=&\ O((1+y)^0)\|g\|_Y+O((1+y)^{-1})\sum_{k=1}^2\sum_{l=3}^4\int_{-1}^y\|g\|_Y\\
        &-\sum_{k=3}^4\sum_{l=3}^4\left[F_L(y)_l^1+iF_L(y)_l^2\right]M(a,p)_k^l\int_y^1\alpha^k(a,p,g)\\
        &+\sum_{k=3}^4\sum_{l=3}^4\left[F_L(y)_l^1+iF_L(y)_l^2\right]M(a,p)_1^l\frac{M(a,p)_k^3}{M(a,p)_1^3}\int_{-1}^1\alpha^k(a,p,g)
        \\&-\sum_{k=3}^4\left[M(a,p)_1^4\frac{M(a,p)^3_k}{M(a,p)_1^3}-M(a,p)_k^4\right]\left[F_L(y)_4^1+iF_L(y)_4^2\right]\int_{-1}^1\alpha^k(a,p,g)\\
        =&\ O((1+y)^0)\|g\|_Y+\sum_{k=3}^4\sum_{l=3}^4\left[F_L(y)_l^1+iF_L(y)_l^2\right]M(a,p)_k^l\int_{-1}^y\alpha^k(a,p,g),
    \end{align*}
    which yields
    \begin{align*}
        \left|\mathcal{J}(a,p,g)(y)\right|\lesssim&\|g\|_Y
    \end{align*}
    for all $(y,a,p)\in(-1,-\tfrac{1}{2}]\times J_*\times I$. Analogously, we find
    \begin{align*}
        \left|\mathcal{J}(a,p,g)'(y)\right|\lesssim\|g\|_Y(1+y)^{-1}
    \end{align*}
    for all $(y,a,p)\in(-1,-\tfrac{1}{2}]\times J_*\times I$. Finally, in the domain $y\in[-\tfrac{1}{2},0]$, the claimed bounds are obvious.
\end{proof}
 
 \noindent
To prove continuity of $\mathcal{J}(.,.,g)$, we will need the following Lemmas to handle the oscillatory behavior near $y=1$.
\begin{lem}\label{osclem}
    Let $J\subset\mathbb{R}\setminus\lbrace 0\rbrace$ be compact and $q\in(1,3)$. Then we have the bound
    $$\left|\int_y^1e^{-i\varphi(x,\alpha)}(1-x)^{q-2}dx\right|\lesssim(1-y)^{q+1}$$
    for all $y\in[0,1)$ and all $\alpha\in J$.
\end{lem}
\begin{proof}
    Recall the definition $\varphi(y,\alpha)=-\tfrac{2\alpha}{(1-y)^2}+\tfrac{2\alpha}{1-y}-\tfrac{2}{\alpha}\log(1-y)$.
    An integration by parts yields
    \begin{align*}
    \int_y^1e^{-i\varphi(x,\alpha)}(1-x)^{q-2}dx=&\int_y^1\frac{i(1-x)^{q-2}}{\varphi'(x,\alpha)}\partial_xe^{-i\varphi(x,\alpha)}dx
    \\=&-\frac{i(1-y)^{q-2}}{\varphi'(y,\alpha)}e^{-i\varphi(y,\alpha)}-\int_y^1e^{-i\varphi(x,\alpha)}\partial_x\left(\frac{i(1-x)^{q-2}}{\varphi'(x,\alpha)}\right)dx    
    \end{align*}
    and by inserting the explicit expression for $\varphi$, the claim follows.
\end{proof}
 
\begin{lem}\label{oscintlem}
    Let $J\subset\mathbb{R}\setminus\lbrace 0\rbrace$ and $K\subset(1,3)$ be compact intervals and $f\in C([0,1])$. Define $\Psi:[0,1)\times J\times K\rightarrow\mathbb{C}$ by
    $$\Psi(y,\alpha,q)=(1-y)^{1-q}e^{i\varphi(y,\alpha)}\int_y^1e^{-i\varphi(x,\alpha)}(1-x)^{q-2}f(x)dx.$$ Then $\Psi$ extends to a continuous function on $[0,1]\times J\times K$.
\end{lem}
\begin{proof}
Let $\epsilon>0$ be arbitrary. Then there exists a $\delta>0$ such that $|f(x)-f(1)|<\epsilon$ for all $x\in[1-\delta,1)$. Lemma \ref{osclem} gives
\begin{align*}
    \left|\Psi(y,\alpha,q)\right|\leq&\biggl|(1-y)^{1-q}\int_y^1e^{-i\varphi(x,\alpha)}(1-x)^{q-2}f(1)dx\biggr|+\biggl|(1-y)^{1-q}\int_y^1e^{-i\varphi(x,\alpha)}(1-x)^{q-2}[f(x)-f(1)]dx\biggr|\\\lesssim&\ (1-y)^2+\epsilon
\end{align*}
for all $y\in[1-\delta,1)$ and all $(\alpha,q)\in J\times K$. Consequently, $\lim_{y\rightarrow 1-}\Psi(y,\alpha,q)=0$ and we extend $\Psi$ to $[0,1]\times J\times K$ by setting $\Psi(1,\alpha,q):=0$.
\end{proof}
 
\begin{lem}\label{JFcontlem}
    The map $(a,p)\mapsto\mathcal{J}(a,p,g):J_*\times I\rightarrow X$ is continuous.
\end{lem}
\begin{proof}
    From the proof of Proposition \ref{Jbdlem} it is evident that $(y,a,p)\mapsto \mathcal{J}(a,p,g)(y)$ extends to a continuous function on $[-1,1]\times J_*\times I$ and $(y,a,p)\mapsto (1-y^2)\mathcal{J}(a,p,g)'(y)$ is continuous on $[-1,1-\delta]\times J_*\times I$ for any $\delta\in (0,1]$. While $f'_1(\cdot,a,p)$, $f'_2(\cdot,a,p)$ extend continuously to $1$, $f'_3(\cdot,a,p)$, $f'_4(\cdot,a,p)$ are oscillatory near $1$. Consequently, the only delicate term that occurs in $\mathcal{J}(a,p,g)'$ is    $$\sum_{k=3}^4\left[F(y,a,p)_k^3+iF(y,a,p)_k^4\right]\int_{y}^1\alpha^k(a,p,g)(x)dx.$$
    Recall that $f_k'(y,a,p)=O((1-y)^{-\tfrac{4}{p-1}}a^0p^0)e^{i\varphi(y,a_*+a)}$ for $k\in\lbrace3,4\rbrace$ and hence, by Lemma \ref{asymFRinvlem}, $$\left[F(y,a,p)_k^3+iF(y,a,p)_k^4\right]\int_y^1\alpha^k(a,p,g)(x)dx$$ behaves like $$(1-y)^{-\tfrac{4}{p-1}}e^{i\varphi(y,a_*+a)}\int_y^1e^{-i\varphi(x,a_*+a)}(1-x)^{\tfrac{4}{p-1}}g(x)dx$$ as $y\rightarrow 1$. By Lemma \ref{oscintlem}, with $K=\lbrace\tfrac{4}{p-1}:p\in I\rbrace\subset(1,3)$, $(y,a,p)\mapsto (1-y^2)\mathcal{J}(a,p,g)'(y)$ extends to a continuous function on $[-1,1]\times J_*\times I$.
\end{proof}
\section{The solution system}
\begin{defin}
    For $(y,a,p)\in (-1,1)\times J_*\times I$ and $f\in X$ we define
    $$\mathcal{G}(a,p,f)(y):=\tilde{\mathcal{L}}_{a,p}(f)(y)-\mathcal{L}_{a,p}(f)(y)-\mathcal{N}_{a,p}(f)(y)-\mathcal{R}(a_*+a,p,f_*(y,a,p)).$$
\end{defin}
\begin{rem}
    Note that $\mathcal{R}(a_*+a,p,f_*(\cdot,a,p)+f)=0$ is equivalent to $\tilde{\mathcal{L}}_{a,p}(f)=\mathcal{G}(a,p,f)$.
\end{rem}
\begin{lem}\label{operatordiffcontlem}
    Let $f\in X$. Then the map $(a,p)\mapsto\tilde{\mathcal{L}}_{a,p}(f)-\mathcal{L}_{a,p}(f):J_*\times I\rightarrow Y$ is continuous.
\end{lem}
\begin{proof}
    By Proposition \ref{coeffFprop}, we have
\begin{align*}
    p_1(y,a,p)-p_0(y,a_*+a,p)&=O((1-y)^{-1}),\\
    p_2(y,a,p)&=O((1-y)^{-1}),\\
    q_1(y,a,p)-q_0(y,a_*+a,p)&=O((1+y)^{-1})+O((1-y)^{-2}),\\
    q_2(y,a,p)&=O((1+y)^{-1})+O((1-y)^{-1})
\end{align*}
and hence, for $f\in X$ we obtain
\begin{align*}
    \left[p_1(y,a,p)-p_0(y,a_*+a,p)\right]f'(y)&=O((1-y)^{-2}(1+y)^{-1}),\\
    p_2(y,a,p)\overline{f'(y)}&=O((1-y)^{-2}(1+y)^{-1}),\\
    \left[q_1(y,a,p)-q_0(y,a_*+a,p)-\frac{p+1}{2}\frac{|f_*(\cdot,a,p)|^{p-1}}{(1-y)^2}\right]f(y)&=O((1+y)^{-1}(1-y)^{-2}),\\
    \left[q_2(y,a,p)-\frac{p-1}{2}\frac{|f_*(\cdot,a,p)|^{p-3}f_*(\cdot,a,p)^2}{(1-y)^2}\right]\overline{f(y)}&=O((1+y)^{-1}(1-y)^{-2}).
\end{align*}
Consequently,
$$\tilde{\mathcal{L}}_{a,p}(f)(y)-\mathcal{L}_{a,p}(f)(y)=O((1+y)^{-1}(1-y)^{-2}),$$
which shows that $\tilde{\mathcal{L}}_{a,p}(f)-\mathcal{L}_{a,p}(f)\in Y$ for all $(a,p)\in J_*\times I$. Moreover, it is obvious from the proof of Proposition \ref{coeffFprop} that the coefficients $p_1,p_2,q_1,q_2$ are continuous on $(-1,1)\setminus\lbrace 0\rbrace\times J_*\times I$ and not oscillatory at the endpoints. In particular, the one-sided limits $\lim_{y\rightarrow\pm 1}(1+y)(1-y)^2\left(\tilde{\mathcal{L}}_{a,p}(f)(y)-\mathcal{L}_{a,p}(f)(y)\right)$ exist and depend continuously on $(a,p)$. Hence, the map $(y,a,p)\mapsto (1+y)(1-y)^2\left(\tilde{\mathcal{L}}_{a,p}(f)(y)-\mathcal{L}_{a,p}(f)(y)\right)$ extends continuously to $[-1,0]\times J_*\times I$ and $[0,1]\times J_*\times I$.
\end{proof}

\begin{lem}\label{nonlincontlem}
    Let $f\in X$. Then the map $(a,p)\mapsto\mathcal{N}_{a,p}(f):J_*\times I\rightarrow Y$ is continuous.
\end{lem}
\begin{proof}
    This follows directly from the continuity of $f_*:[-1,1]\times J_*\times I\rightarrow\mathbb{C}$.
\end{proof}
\begin{lem}\label{nonlinestlem}
    Let $(a,p)\in J_*\times I$. Then
    \begin{align*}
        \|\mathcal{N}_{a,p}(f)-\mathcal{N}_{a,p}(g)\|_Y\leq&\ (p-1)\cdot\big((p+1)+\max\lbrace2,p-1\rbrace\big)\cdot\max\lbrace\|f\|_X,\|g\|_X\rbrace\\&\cdot\left(\|f_*(\cdot,a,p)\|_{L^\infty}+\max\lbrace\|f\|_X,\|g\|_X\rbrace\right)^{p-2}\cdot\|f-g\|_X
    \end{align*}
   In particular, the bound depends continuously on $p$.
   \begin{align*}
        \|\mathcal{N}_{a,3}(f)-\mathcal{N}_{a,3}(g)\|_Y\leq&12\cdot\max\lbrace\|f\|_X,\|g\|_X\rbrace\cdot\big(\tfrac{27}{2}+\max\lbrace\|f\|_X,\|g\|_X\rbrace \big)\cdot\|f-g\|_X
    \end{align*}
   for all $a\in J_*$.
\end{lem}
\begin{proof}
    Let $G(z)=|z|^{p-1}z$ and $DG(f_0)\cdot f=\tfrac{p+1}{2}|f_0|^{p-1}f+\tfrac{p-1}{2}|f_0|^{p-3}f_0^2\cdot\Bar{f}$, then
    \begin{align*}
        (1-y)^2\left[\mathcal{N}_{a,p}(f)(y)-\mathcal{N}_{a,p}(g)(y)\right]=&\ G(f_*(\cdot,a,p)+f)(y)-G(f_*(\cdot,a,p)+g)(y)-DG(f_*(\cdot,a,p))\cdot(f-g)(y)\\
        =&\int_0^1 \bigg[DG(f_*(\cdot,a,p)+g+t(f-g))-DG(f_*(\cdot,a,p))\bigg]\cdot (f-g)(y)\ dt.
    \end{align*}
    Now $\partial_z\bigg[|z|^{p-1}\bigg]=\tfrac{p-1}{2}|z|^{p-3}\Bar{z}$, $\partial_{\Bar{z}}\bigg[|z|^{p-1}\bigg]=\tfrac{p-1}{2}|z|^{p-3}z$ and $\partial_z\bigg[|z|^{p-3}z^2\bigg]=\tfrac{p+1}{2}|z|^{p-1}\Bar{z}^{-1}$, $\partial_{\Bar{z}}\bigg[|z|^{p-3}z^2\bigg]=\tfrac{p-3}{2}|z|^{p+1}\Bar{z}^{-3}$ gives
    \begin{align*}
        \left||z|^{p-1}z-|w|^{p-1}w\right|\leq& (p-1)\cdot|z-w|\cdot\max\lbrace|z|,|w|\rbrace^{p-2},\\
        \left||z|^{p-3}z^2-|w|^{p-3}w^2\right|\leq&\max\lbrace2,p-1\rbrace\cdot|z-w|\cdot\max\lbrace|z|,|w|\rbrace^{p-2},
    \end{align*}
    and hence,
    \begin{align*}
        |(DG(z)-DG(w))\cdot(f-g)|=&\ |\tfrac{p+1}{2}\left[|z|^{p-1}z-|w|^{p-1}w\right](f-g)+\tfrac{p-1}{2}\left[|z|^{p-3}z^2-|w|^{p-3}w^2\right](\Bar{f}-\Bar{g})|\\
        \leq&\left[\tfrac{(p+1)(p-1)}{2}+\tfrac{p-1}{2}\max\lbrace2,p-1\rbrace\right]\max\lbrace|z|,|w|\rbrace^{p-2}\cdot|z-w|\cdot|f-g|.
    \end{align*}
    For $z=f_*(\cdot,a,p)+g+t(f-g)$, $w=f_*(\cdot,a,p)$ we obtain
    \begin{align*}
        (1+y)(1-y)^2|\mathcal{N}_{a,p}(f)(y)-\mathcal{N}_{a,p}(g)(y)|\leq&\ (1+y)\left[\tfrac{(p+1)(p-1)}{2}+\tfrac{p-1}{2}\max\lbrace2,p-1\rbrace\right]\cdot\max\lbrace\|f\|_X,\|g\|_X\rbrace\\&\cdot\left(\|f_*(\cdot,a,p)\|_{L^\infty}+\max\lbrace\|f\|_X,\|g\|_X\rbrace\right)^{p-2}\cdot\|f-g\|_X\\
        \leq&\big((p+1)(p-1)+(p-1)\cdot\max\lbrace2,p-1\rbrace\big)\cdot\max\lbrace\|f\|_X,\|g\|_X\rbrace\\&\cdot\left(\|f_*(\cdot,a,p)\|_{L^\infty}+\max\lbrace\|f\|_X,\|g\|_X\rbrace\right)^{p-2}\cdot\|f-g\|_X.\\
    \end{align*}
    Continuity in $p$ follows directly from the continuity of $f_*$.
\end{proof}
\begin{lem}\label{GFcontlem}
    Let $f\in X$. Then, the map $(a,p)\mapsto \mathcal{G}(a,p,f): J_*\times I\rightarrow Y$ is continuous.
\end{lem}
\begin{proof}
    This follows directly from Lemmas \ref{Rcontlem}, \ref{operatordiffcontlem} and \ref{nonlincontlem}.
\end{proof}
The following lower and upper bound for the approximate solution $f_*(\cdot,a,3)$ is the only part of the proof that requires computer assistance.
 \begin{lem}\label{approxBD}
    The function $f_*(\cdot,a,3)$ satisfies the bounds
    \begin{align*}
        \|f_*(\cdot,a,3)\|_{L^\infty([-1,1])}&\leq \frac{27}{2},\\
        \min_{y\in[-1,1]}|f_*(y,a,3)|&\geq \frac{19}{10}
    \end{align*}
    for all $a\in J_*$. In particular, $f_*(\cdot,a,3)$ has no zeros.
\end{lem}
\begin{proof}
    First note, that $\|T_n'\|_{L^\infty}\leq n^2$ for all $n$ implies the bound $\|p\|_{L^\infty}\leq\sum_{n=0}^Nn^ 2|c_n|$ for any Chebychev polynomial $p=\sum_{n=0}^Nc_nT_n$. This yields the bound $$\|P_*'\|_{L^\infty}\leq12.$$
    Let $\Omega=\lbrace-1+k\cdot\tfrac{2}{100},\ k\in[0,100]\cap\mathbb{Z}\rbrace$. Evaluation yields
    $$\max_\Omega|P_*|\leq \frac{27}{10},\hspace{1cm}\min_\Omega|P_*|\geq \frac{11}{10}$$
    and hence, by the mean value theorem
    \begin{align*}
        \|P_*\|_{L^\infty([-1,1])}&\leq \frac{27}{10}+\frac{2}{100}\cdot 12<3,\\
        \min_{y\in[-1,1]}|P_*|&\geq \frac{11}{10}-\frac{2}{100}\cdot 12>\frac{8}{10}.
    \end{align*}
    Now recall
    $f_*(y,a,3)=P_*(y)+\tfrac{4i(a_*+a)}{2-2i(a_*+a)}P_*'(1)-P_*(1)$. It is easy to see that 
    $$\left|\frac{4i(a_*+a)}{2-2i(a_*+a)}\right|=\left|-2+\frac{2}{1+(a_*+a)^2}+i\frac{2}{1+(a_*+a)^2}\right|\in \left[1,\frac{3}{2}\right]$$ for all $a_*\in[-10^{-10},10^{-10}]$. Evaluating
    $$|P_*(1)|\leq \frac{6}{5},\hspace{1cm}\frac{61}{10}\leq|P_*'(1)|\leq \frac{31}{5}$$
    we obtain
    \begin{align*}
        \|f_*(\cdot,a,3)\|_{L^\infty([1,1])}&\leq 3+\frac{3}{2}\cdot \frac{31}{5}+\frac{6}{5}=\frac{27}{2},\\
        \min_{y\in[-1,1]}|f_*(y,a,3)|&\geq \frac{61}{10}\cdot 1-3-\frac{6}{5}=\frac{19}{10}
    \end{align*}
    for all $a\in J_*$.
\end{proof}
\begin{prop}\label{contractionprop}
There exists an open interval $\mathring{I}_0\subset I\subset (\tfrac{7}{3},5)$ such that $3\in \mathring{I}_0$ and for all $(a,p)\in J_*\times I_0$, where $I_0$ denotes the closure of $\mathring{I}_0$, there exists a unique $f_{a,p}\in X$ with $\|f_{a,p}\|_X\leq 1.2 \cdot 10^{-6}$ such that
$$f_{a,p}(y)=\mathcal{J}(a,p,\mathcal{G}(a,p,f_{a,p}))(y)$$
for all $y\in(-1,1)$. Furthermore, the map $(a,p)\mapsto f_{a,p}:J_*\times I_0\rightarrow X$ is continuous.
\end{prop}
\begin{proof}
    Let $C_\mathcal{J}(p):=\sup\lbrace\|\mathcal{J}(a,p,g)\|_X:g\in Y, \|g\|_Y\leq 1, a\in J_*\rbrace$ and note that from the proof of Lemma \ref{JFcontlem}, it is obvious that $p\mapsto C_\mathcal{J}(p)$ is continuous $I\rightarrow\mathbb{R}$. By Lemmas \ref{operatordiffcontlem}, \ref{Rcontlem} and \ref{nonlinestlem} we obtain that the bound
\begin{align*}
    \|\mathcal{J}(a,p,\mathcal{G}(a,p,f))\|_X\leq&\ C_\mathcal{J}(p)\cdot \|\mathcal{G}(a,p,f)\|_Y\\
    \leq&\ C_\mathcal{J}(p)\cdot\left(\|\tilde{\mathcal{L}}_{a,p}(f)-\mathcal{L}_{a,p}(f)\|_Y+\|\mathcal{R}(a_*+a,p,f_*(y,a,p))\|_Y+\|\mathcal{N}_{a,p}(f)\|_Y\right)\\
    \leq&\ C_\mathcal{J}(p)\cdot\bigg(\|\tilde{\mathcal{L}}_{a,p}(f)-\mathcal{L}_{a,p}(f)\|_Y+\|\mathcal{R}(a_*+a,p,f_*(y,a,p))\|_Y\\
    &\phantom{\ C_\mathcal{J}(p)\cdot(\ }+(p-1)((p+1)+\max\lbrace 2,p-1\rbrace)\cdot(\|f_*(\cdot,a,p)\|_{L^\infty}+\|f\|_X)^{p-2}\cdot\|f\|_X^2\bigg)
\end{align*}
depends continuously on $p$.
Evaluation at $p=3$, using the bound for $\|f_*(\cdot,a,3)\|_{L^\infty}$ from Lemma \ref{approxBD} and
$$C_\mathcal{J}(3)\leq 232,\hspace{0.5cm}\|\tilde{\mathcal{L}}_{a,3}(f)-\mathcal{L}_{a,3}(f)\|_Y\leq5\cdot 10^{-6}\|f\|_X,\hspace{0.5cm}\|\mathcal{R}(a_*+a,3,f_*(y,a,3))\|_Y\leq 4\cdot 10^{-9},$$
for all $a\in J_*$ from Propositions 6.4, 7.6 and Lemmas 8.1, 8.9 in \cite{DS}, yields
\begin{align*}
\|\mathcal{J}(a,3,\mathcal{G}(a,3,f))\|_X\leq&\
    232\cdot\bigg(6\cdot 10^{-6}\|f\|_X+4\cdot 10^{-9}+12\cdot\left(\tfrac{27}{2}+\|f\|_X\right)\cdot\|f\|_X^2\bigg)
    <\ 1.2\cdot10^{-6}
\end{align*}
for all $f\in X$ satisfying $\|f\|_X\leq 1.2\cdot 10^{-6}$ and $a\in J_*$. Hence, by continuity
$$\|\mathcal{J}(a,p,\mathcal{G}(a,p,f))\|_X<1.2\cdot10^{-6}$$
for all $f\in X$ satisfying $\|f\|_X\leq 1.2\cdot 10^{-6}$ and $(a,p)\in J_*\times I_0$ with $I_0$ sufficiently small around $3$.\\
Similarly we have
\begin{align*}
    &\|\mathcal{J}(a,p,\mathcal{G}(a,p,f))-\mathcal{J}(a,p,\mathcal{G}(a,p,g))\|_X\leq\ C_\mathcal{J}(p)\cdot\|\mathcal{G}(a,p,f)-\mathcal{G}(a,p,g)\|_Y\\
    &\leq C_\mathcal{J}(p)\cdot\bigg(\|\tilde{\mathcal{L}}_{a,p}(f-g)-\mathcal{L}_{a,p}(f-g)\|_Y+\|\mathcal{N}_{a,p}(f)-\mathcal{N}_{a,p}(g)\|_Y\bigg)\\
    &\leq C_\mathcal{J}(p)\cdot\bigg(\|\tilde{\mathcal{L}}_{a,p}(f-g)-\mathcal{L}_{a,p}(f-g)\|_Y+\ (p-1)\cdot\big((p+1)+\max\lbrace2,p-1\rbrace\big)\cdot\max\lbrace\|f\|_X,\|g\|_X\rbrace\\&\phantom{\leq C_\mathcal{J}(p)\cdot\bigg(}\cdot\left(\|f_*(\cdot,a,p)\|_{L^\infty}+\max\lbrace\|f\|_X,\|g\|_X\rbrace\right)^{p-2}\cdot\|f-g\|_X\bigg)
\end{align*}
and evaluation at $p=3$ yields
\begin{align*}
&\|\mathcal{J}(a,3,\mathcal{G}(a,3,f))-\mathcal{J}(a,3,\mathcal{G}(a,3,g))\|_X\\&\leq
    232\cdot\bigg(6\cdot 10^{-6}+12\cdot\max\lbrace\|f\|_X,\|g\|_X\rbrace\cdot(\tfrac{27}{2}+\max\lbrace\|f\|_X,\|g\|_X\rbrace)\bigg)\cdot\|f-g\|_X\\
    &<\ \tfrac{1}{2}\|f-g\|_X
\end{align*}
for all $f,g\in X$ satisfying $\|f\|_X,\|g\|_X\leq 1.2\cdot 10^{-6}$ and $a\in J_*$. Hence, by continuity
$$\|\mathcal{J}(a,p,\mathcal{G}(a,p,f))-\mathcal{J}(a,p,\mathcal{G}(a,p,g))\|_X<\ \tfrac{1}{2}\|f-g\|_X$$
for all $f,g\in X$ satisfying $\|f\|_X,\|g\|_X\leq 1.2\cdot 10^{-6}$ and $(a,p)\in J_*\times I_0$ with $I_0$ sufficiently small around $3$.
    Consequently, for all $(a,p)\in J_*\times I_0$, $f\mapsto \mathcal{J}(a,p,\mathcal{G}(a,p,f))$ is a contraction on the ball $\lbrace f\in X: \|f\|_X\leq 1.2 \cdot 10^{-6}\rbrace$ and there exists a unique fixed point $f_{a,p}$.
    
    For the continuity, note that
    \begin{align*}
        \|f_{a,p}-f_{b,q}\|_X=&\ \|\mathcal{J}(a,p,\mathcal{G}(a,p,f_{a,p})-\mathcal{J}(b,q,\mathcal{G}(b,q,f_{b,q})\|_X\\
        \leq&\ \|\mathcal{J}(a,p,\mathcal{G}(a,p,f_{a,p})-\mathcal{J}(b,q,\mathcal{G}(b,q,f_{a,p})\|_X+ \tfrac{1}{2}\|f_{a,p}-f_{b,q}\|_X
    \end{align*}
    and the claim follows from the continuity of $(b,q)\mapsto\mathcal{J}(b,q,\mathcal{G}(b,q,f_{a,p})):J_*\times I_0\rightarrow X$.
\end{proof}

\begin{prop}\label{psizeroprop}
    Let $(a,p)\mapsto f_{a,p}:J_*\times I_0\rightarrow X$ be the continuous map constructed in Proposition \ref{contractionprop}. Then there exists an open interval $\mathring{I}_*\subset I_0$ such that $3\in \mathring{I}_*$ and for all $p\in I_*$, where $I_*$ denotes the closure of $\mathring{I}_*$, there exists an $a_p\in J_*$ such that
    $$\psi(a_p,p,\mathcal{G}(a_p,p,f_{a_p,p}))=0.$$
\end{prop}
\begin{proof}
Let $a_\pm:=\pm10^{-10}$. From the proof of Proposition 10.3 in \cite{DS} we have
$$\psi(a_-,3,\mathcal{G}(a_-,3,f_{a_-,3}))<0,\hspace{1cm}\psi(a_+,3,\mathcal{G}(a_+,3,f_{a_+,3}))>0.$$
By continuity of the map $p\mapsto \psi(a,p,\mathcal{G}(a,p,f_{a,p})):I_*\rightarrow\mathbb{R}$ (Lemma \ref{psicontlem}) we conclude that by choosing $I_*\subseteq I_0$ sufficiently small around 3,
$$\psi(a_-,p,\mathcal{G}(a_-,p,f_{a_-,p}))<0,\hspace{1cm}\psi(a_+,p,\mathcal{G}(a_+,p,f_{a_+,p}))>0$$
for all $p\in I_*$.
Applying the intermediate value theorem to the continuous map $a\mapsto \psi(a,p,\mathcal{G}(a,p,f_{a,p}))$ yields an $a\in J_*$ with the claimed property.
\end{proof}

\begin{cor}\label{apfpCor}
    Let $I_*$ be the interval described in Proposition \ref{psizeroprop}. Then for all $p\in I_*$, there exist $(a_p,f_p)\in J_*\times X\cap C^2(-1,1)$ with $\|f\|_X\leq 1.2 \cdot 10^{-6}$ such that
    $$\mathcal{R}(a_*+a_p,p,f_*(\cdot,a_p,p)+f_p)=0$$
    on $(-1,1)$.
\end{cor}
\begin{proof}
    Let $p\in I_*$. Then by Propositions \ref{contractionprop} and \ref{psizeroprop} there exist $(a_p,f_p)\in J_*\times X$ such that $f_p(y)=\mathcal{J}(a_p,p,\mathcal{G}(a_p,p,f_p))(y)$ for all $y\in(-1,1)$ and $\psi(a_p,p,\mathcal{G}(a_p,p,f_p))=0$. Since $\psi(a_p,p,\mathcal{G}(a_p,p,f_p))=0$, we have $\mathcal{J}(a_p,p,\mathcal{G}(a_p,p,f_p))=\tilde{\mathcal{L}}_{a_p,p}^{-1}(\mathcal{G}(a_p,p,f_p))$. By Lemma \ref{GFcontlem}, $\mathcal{G}(a_p,p,f_p)\in Y$, and hence, by Lemma \ref{Linvlem}, $f_p\in C^2((-1,1)\setminus\lbrace 0\rbrace)\cap C^1(-1,1)$ and
    $$\tilde{\mathcal{L}}_{a_p,p}(f)(y)=\tilde{\mathcal{L}}_{a_p,p}(\tilde{\mathcal{L}}_{a_p,p}^{-1}(\mathcal{G}(a_p,p,f_p)))(y)=\mathcal{G}(a_p,p,f_p)(y)$$ for all $y\in(-1,1)\setminus\lbrace 0 \rbrace$, or equivalently
    $$\mathcal{R}(a_*+a_p,p,f_*(\cdot,a_p,p)+f_p)=0$$
    on $(-1,1)\setminus\lbrace 0\rbrace$. Now, since $f_p\in C^1(-1,1)$ and $\mathcal{R}(a_*+a_p,p,f_*(\cdot,a_p,p)+f_p)=0$ is a second order differential equation with coefficients and nonlinearity in $C(-1,1)$, it follows that $f_p\in C^2(-1,1)$ and $\mathcal{R}(a_*+a,p,f_*(\cdot,a,p)+f_p)=0$ on all of $(-1,1)$.
\end{proof}

\noindent It remains to prove the stated regularity properties.
\begin{lem}\label{finiteElem}
    Let $f\in X, \alpha\in\mathbb{R}\setminus\lbrace 0\rbrace$, $p\in(\tfrac{7}{3},5)$, and $h:\mathbb{R}^3\rightarrow\mathbb{C}$ be defined by
    $$h(x):=(1+|x|)^{-\tfrac{2}{p-1}-\tfrac{i}{\alpha}}f\left(\frac{|x|-1}{|x|+1}\right).$$
    Then $h\in L^{p+1}(\mathbb{R}^3)\cap\dot{H}^1(\mathbb{R}^3)$.
\end{lem}
\begin{proof}
    We define $\tilde{h}:[0,\infty)\rightarrow\mathbb{C}$ by
    $$\tilde{h}(r):=(1+r)^{-\tfrac{2}{p-1}-\tfrac{i}{\alpha}}f\left(\frac{r-1}{r+1}\right).$$
    Then $h(x)=\tilde{h}(|x|)$ and
    $$\|h\|^{p+1}_{L^{p+1}(\mathbb{R}^3)}\simeq\int_0^\infty|\tilde{h}(r)|^{p+1}r^2dr,\hspace{1cm}\|h\|^{2}_{\dot{H}^1(\mathbb{R}^3)}\simeq\int_0^\infty|\tilde{h}'(r)|^2r^2dr.$$
    The change of variable $r=\frac{1+y}{1-y}$ yields $dr=\tfrac{2}{(1-y)^2}dy$ and thus,
    $$\|h\|^{p+1}_{L^{p+1}(\mathbb{R}^3)}\simeq\int_{-1}^1(1-y)^{\tfrac{2(p+1)}{p-1}}|f(y)|^{p+1}\left(\frac{1+y}{1-y}\right)^2(1-y)^{-2}dy\lesssim\|f\|_X^{p+1}$$
    and
    \begin{align*}
        \|h\|^{2}_{\dot{H}^1(\mathbb{R}^3)}\simeq&\int_{-1}^1(1-y)^{\tfrac{2(p+1)}{p-1}}|f(y)|^2\left(\frac{1+y}{1-y}\right)^2(1-y)^{-2}dy\\&+\int_{-1}^1(1-y)^{\tfrac{4}{p-1}+4}|f'(y)|^2\left(\frac{1+y}{1-y}\right)^2(1-y)^{-2}dy\\
        \lesssim&\|f\|_X^2.
    \end{align*}
\end{proof}

\begin{lem}\label{Qpreg}
For $p\in I_*$, let $(a_p,f_{p})\in J_*\times X\cap C^2(-1,1)$ be as in Corollary \ref{apfpCor} and define $Q_p:\mathbb{R}^3\rightarrow\mathbb{C}$ by
    $$Q_p(x):=(1+|x|)^{-\tfrac{2}{p-1}-\tfrac{1}{a_*+a_p}}\left[f_*\left(\frac{|x|-1}{|x|+1},a_p,p\right)+f_{p}\left(\frac{|x|-1}{|x|+1}\right)\right].$$
    Then $Q_p\in L^{p+1}(\mathbb{R}^3)\cap\dot{H}^1(\mathbb{R}^3)\cap C^\infty(\mathbb{R}^3)$ and $Q_p$ satisfies Eq.~(\ref{QEq}) on $\mathbb{R}^3$.
\end{lem}
\begin{proof}
    By construction and Lemma \ref{finiteElem}, $Q_p\in L^{p+1}(\mathbb{R}^3)\cap\dot{H}^1(\mathbb{R}^3)\cap C^2(\mathbb{R}^3\setminus\lbrace 0\rbrace)$ and $Q_p$ satisfies Eq.~(\ref{QEq}) on $\mathbb{R}^3\setminus\lbrace 0\rbrace$.
    We set $F_p(x):=Q_p-i\alpha_p\left[x\cdot\nabla Q_p(x)+\tfrac{2}{p-1}Q_p(x)\right]-Q_p(x)|Q_p(x)|^{p-1}$. Since obviously $Q_p$ is bounded, we have $Q_p,|Q_p|^{p-1}Q_p\in L^p_{\mathrm{loc}}(\mathbb{R}^3)$ for all $p\geq 1$. In particular, $F_p\in L^2_{\mathrm{loc}}(\mathbb{R}^3)=W^{0,2}_{\mathrm{loc}}(\mathbb{R}^3)$ and hence, $Q_p\in W^{2,2}_{\mathrm{loc}}(\mathbb{R}^3)$. By the Sobolev embedding $W^{2,2}_{\mathrm{loc}}(\mathbb{R}^3)\hookrightarrow W^{1,6}_{\mathrm{loc}}(\mathbb{R}^3)$, we have $(\cdot)\cdot\nabla Q_p\in L^6_{\mathrm{loc}}(\mathbb{R}^3)$ and hence $F_p\in L^6_{\mathrm{loc}}(\mathbb{R}^3)=W^{0,6}_{\mathrm{loc}}(\mathbb{R}^3)$ and $Q_p\in W^{2,6}_\mathrm{loc}$. By the Sobolev embedding $W^{2,6}_{\mathrm{loc}}(\mathbb{R}^3)\hookrightarrow C^1_{\mathrm{loc}}(\mathbb{R}^3)$, we have $F_p\in C_{\mathrm{loc}}(\mathbb{R}^3)$ and hence, $Q_p\in C^2_{\mathrm{loc}}(\mathbb{R}^3)$. From Lemma \ref{approxBD} and the continuity of $p\mapsto f_*(\cdot,a,p)$ it follows that $Q_p$ is non-vanishing for all $p\in I_*$ by choosing $I_*$ sufficiently small around $3$. Hence, by bootstrapping we obtain $Q_p\in\cap_{k\geq 1}C^k_{\mathrm{loc}}(\mathbb{R}^3)=C^\infty(\mathbb{R}^3)$ and $Q_p$ satisfies Eq.~(\ref{QEq}) on all of $\mathbb{R}^3$.
\end{proof}

We conclude the proof of Theorem \ref{quantThm} by proving the stated estimate on the function $g_p$.
\begin{lem}
    Let $f\in X$ and define $g:[0,\infty)\rightarrow\mathbb{C}$ by
    $$g(r):=f\left(\frac{r-1}{r+1}\right).$$
    Then we have
    $$2\sup_{r>0}r|g'(r)|+\sup_{r>0}|g(r)|\leq\|f\|_X.$$
\end{lem}
\begin{proof}
    By definition, we have
    $$g'(r)=\frac{2}{(r+1)^2}f'\left(\frac{r-1}{r+1}\right)$$
    and thus,
    $$|2rg'(r)|=\left|\frac{4r}{(r+1)^2}f'\left(\frac{r-1}{r+1}\right)\right|\leq\sup_{y\in(-1,1)}(1-y^2)|f'(y)|$$
    for all $r>0$.
\end{proof}

Theorem \ref{qualThm} is a direct implication of Theorem \ref{quantThm}.

\bibliographystyle{plain}
\bibliography{sample}

\end{document}